\documentclass[a4paper,reqno]{amsart}  
\usepackage{graphicx}
\usepackage[]{geometry}
\usepackage[english]{babel}
\usepackage{amsfonts, color}
\usepackage{amstext, amssymb, amsmath}
\usepackage{amssymb}
\usepackage{algorithm, algorithmic}
\usepackage{soul}
\usepackage{mathrsfs}

\newcommand{\R}{\mathbb{R}}

\newcommand{\subel}{\mathfrak{v}}
\newcommand{\faux}{f^{\mathrm{aux}}}
\newcommand{\sign}{\mathrm{sign}}
\newcommand{\xtemp}{x^\mathrm{temp}}

\theoremstyle{plain}
\begingroup
\newtheorem{theorem}{Theorem}[section]
\newtheorem{lemma}[theorem]{Lemma}
\newtheorem{proposition}[theorem]{Proposition}

\endgroup

\theoremstyle{definition}
\begingroup
\newtheorem{defn}{Definition}

\endgroup

\theoremstyle{remark}
\newtheorem{remark}{Remark}

\numberwithin{equation}{section}

\begin{document}
\title[A subgradient method for $\ell_1$-composite optimization]
{A subgradient method with constant step-size for $\ell_1$-composite optimization}

\author[A. Scagliotti and P. Colli Franzone]{Alessandro Scagliotti and Piero Colli Franzone}

\address[A.~Scagliotti]{Technical University of Munich (TUM) \& Munich Center for Machine Learning (MCML), Germany}
\email{scag@ma.tum.de}

\address[P.~Colli~Franzone]{Dipartimento di Matematica, Universit\`a di Pavia, Italy}
\email{piero.collifranzone@unipv.it }

\begin{abstract}
Subgradient methods are the natural extension to the non-smooth case of the classical gradient descent for regular convex optimization problems. However, in general, they are characterized by slow convergence rates, and they require decreasing step-sizes to converge.
In this paper we propose a subgradient method with constant step-size for composite convex objectives with $\ell_1$-regularization.
If the smooth term is strongly convex, we can establish a linear convergence result for the function values. This fact relies on an accurate choice of the element of the subdifferential used for the update, and on proper actions adopted when non-differentiability regions are crossed.
Then, we propose an accelerated version of the algorithm, based on conservative inertial dynamics and on an adaptive restart strategy, that is guaranteed to achieve a linear convergence rate in the strongly convex case.
Finally, we test the performances of our algorithms on some strongly and non-strongly convex examples.\\


\end{abstract}

\maketitle

\section*{Introduction}
In this paper we deal with convex \textit{composite} optimization, i.e., we consider objective functions $f:\R^n\to \R$ of the form
\begin{equation} \label{eq:Int_comp}
f(x) = g(x) + h(x),
\end{equation}
where $g:\R^n\to\R$ is $C^1$-regular with Lipschitz-continuous gradient, and $h:\R^n\to\R \cup \{+\infty\}$ is a non-smooth convex function.
We recall that the concept of \textit{composite} function was introduced by Nesterov in \cite{N13}, and it usually denotes the splitting \eqref{eq:Int_comp} in the case that the non-regular term $h$ is \textit{simple}. In this framework, possible examples of \textit{simple} functions include, e.g., the indicator of a closed convex set, or the supremum of a finite family of linear functions. The problem of minimizing such composite functions can be effectively addressed by means of \textit{forward-backward} methods (see, e.g., \cite{CW05}), and their accelerated versions \cite{BT}. In this regard, we report the recent contribution \cite{ReCa}, where it is considered an accelerated method that achieves linear convergence when $g,h$ in \eqref{eq:Int_comp} are strongly convex.

The aim of this paper is to develop a convergent subgradient method with constant step-size for the minimization of particular instances of \eqref{eq:Int_comp}.
The subgradient method was first introduced in \cite{Sh68} and, given an initial guess $x^0\in \R^n$, the algorithm produces a sequence $(x^k)_{k\geq0}$ with update rule
\begin{equation} \label{eq:Int_subg_method}
x^{k+1} = x^k - h_k\subel^k
\quad k\geq0,
\end{equation}
where $\subel^k\in \partial f(x^k)$, i.e., it is an element taken from the subdifferential of the objective at the point $x^k$, and $h_k>0$ denotes the step-size.
If we set $\nu_k = h_k|\subel^k|_2$, we can equivalently rephrase \eqref{eq:Int_subg_method} as
\begin{equation} \label{eq:Int_subg_steplength}
x^{k+1} = x^k - \nu_k\frac{\subel^k}{|\subel^k|_2},
\end{equation}
where $\nu_k$ represents the \textit{step-length} at the $k$-th iteration. 
It is possible to deduce the convergence $\lim_{k\to\infty}f(x^k) = f(x^*)$ as soon as $(\nu_k)_{k\geq0}$ satisfies $\lim_{k\to\infty}\nu_k=0$ and $\sum_{k=1}^\infty \nu_k=\infty$ (see \cite[Chapter~2]{Sh98}). In \cite[Theorem~5.2]{Po87} it is proposed a construction for $(\nu_k)_{k\geq 1}$ that achieves $f(x^k)-f(x^*) = o(1/\sqrt{k})$ as $k\to\infty$ when the value $f(x^*)$ is known a priori. We insist on the fact that, in the results mentioned above, the vector $\subel^k$ can be \textit{any} element of $\partial f(x^k)$.  
If we now consider \textit{constant} step-sizes, i.e.,  $h_k=h>0$ for every $k\geq 0$, in general we cannot expect the convergence of the iterates of \eqref{eq:Int_subg_method} to a minimizer. For instance, given the one-dimensional function $f:x\mapsto |x|$, for every choice $h>0$, if the initial guess $x_0\not\in \{ rh: r\in \mathbb{Z} \}$, then the sequence produced by \eqref{eq:Int_subg_method} oscillates and it remains well-separated from $0$. From this example it is clear that, in order to work out a convergent subgradient method with constant step-size, it is crucial to identify the regions where the objective $f$ is non-differentiable, and to take proper actions when the sequence $(x^k)_{k\geq0}$ crosses them. 
Moreover, in our analysis a role of primary importance is played by the choice of the element $\subel^k \in \partial f(x^k)$ used for the iteration. \\
Subgradient methods with constant step-size have already been considered in the convex optimization literature, and, typically, it is possible to prove that the iterates arrive to the sublevel set $\{ x\in \R^n: f(x) \leq \inf f + c \}$, where the quantity $c>0$ is related to the step-size $h>0$. In a similar flavor, if the objective function is strongly convex, the sequence produced by the algorithm manages to reach a ball centered at the minimizer, whose radius depends on $h$. For a presentations of these results, we refer the reader to \cite[Section~3.2]{Ber15}. 
Moreover, under suitable assumptions on the growth of $f$ around the minimizer $x^*$, it is possible to prove that the distance of the iterates to $x^*$ has a linear decay, up to a certain threshold (that, once again, is estimated in function of $h$). For further details, see \cite[Theorem~1]{JM20} and \cite[Theorem~4.3]{DDMP18}.
Finally, we report the recent contribution of \cite{JL23}, where the authors study the stability of a subgradient method with constant step-size around local minimizers, when $f$ is non-smooth and non-convex.
To the best of our knowledge, the one presented here are the first convergence results for a subgradient method with constant step-size.

In this paper, we devote our attention to the case where the non-regular term at the right-hand side of \eqref{eq:Int_comp} consists in the $\ell_1$-penalization, i.e., where we have $h(x)=\gamma |x|_1=\gamma\sum_{i=1}^n|x_i|$ with $\gamma>0$, and
\begin{equation*}
f(x) = g(x) + \gamma |x|_1.
\end{equation*}
This kind of problem is well-studied since the presence of the $\ell_1$-norm induces sparsity in the minimizer, and for this reason such minimization tasks easily arise in real-world applications.
For instance, we recall \cite{CRT08} for signal processing applications, {\cite{YW10} for imaging problems, and finally \cite{ED,VBL} for the $\ell_1$-regularized logistic regression, which is widely used in machine learning, computer vision, data mining, bioinformatics and neural signal processing.}

In our approach, we take advantage of the structure of the points where the objective $f$ is non-differentiable. We recall that, in the case of $\ell_1$-penalization, such points coincide with the set $\bigcup_{i=1}^n\{ x\in \R^n : x_i=0 \}$. Hence, at each iteration, if the current value $x^k$ has some null component, i.e., $x^k\in \bigcap_{i\in \beta_{x^k}}\{ x\in \R^n: x_i=0 \}$ for some $\beta_{x^k}\subset\{ 1,\ldots,n \}$, we first decide which hyperplanes $\{ x\in \R^n: x_i=0 \}$ $i\in \beta_{x^k}$ we move parallel to. This choice is authomatically done by selecting for the update \eqref{eq:Int_subg_method} the direction $\subel^k = \partial^-f(x^k)\in \partial f(x^k)$, where $\partial^-f(x^k)$ denotes the element of $\partial f(x^k)$ with minimal Euclidean norm. 
The interesting situation occurs when some components \textit{strictly} change sign when moving from $x^k$ to $x^{k} - h\partial^-f(x^k)$. In that case, we have to properly decide whether to allow (some of) these changes of sign, or to set the corresponding components equal to 0. We stress the fact that this phase is fundamental in order to avoid the oscillations that characterized the one-dimensional example reported above.
For this method, described in Algorithm~\ref{Alg_subgr_l1}, we can establish a linear convergence result as soon as the regular function $g:\R^n\to\R$ appearing at the right-hand side of \eqref{eq:Int_comp} is \textit{strongly convex}. 
To show that, we make use of a non-smooth version of the Polyak-Lojasiewicz inequality (see, e.g., \cite{BNP17, XY13}).\\
Then, in Section~\ref{sec:acc_l1}, we propose a momentum-based acceleration of Algorithm~\ref{Alg_subgr_l1}, inspired by the \textit{restarted-conservative algorithm} introduced in \cite{ScCF22}. In the smooth convex framework, the idea of introducing momentum to accelerate the convergence of the classical gradient method dates back to the 1960s, with the works of Polyak \cite{P63,P64}. These methods, often called \textit{heavy-ball}, can be interpreted as discretization of a second order damped mechanical system, where the objective function plays the role of the potential energy.
In \cite{SBC} it was shown that also the celebrated Nesterov accelerated gradient method (see \cite{N83}) can be interpreted in this framework.
This led to a renewed interest in the interplay between discrete-time optimization algorithms and continuous-time dynamical models. In this context, in the mechanical system, the classical linear and isotropic viscosity friction is often replaced by a more general dissipative term. In this regard, we recall the contributions \cite{A16,A18,SJ19}. From the discrete-time side, in \cite{OC} the authors empirically observed that adaptively resetting to $0$ the momentum variable (i.e., the velocity) can further boost the convergence.
Motivated by this fact, in \cite{ScCF22} it was considered a \textit{conservative} dynamical model (i.e., without any dissipative term in the dynamics), whose convergence \textit{completely} relies on a proper restart scheme. In Algorithm~\ref{Alg_acc_subgr_l1} we propose for composite functions with $\ell_1$-penalization a new version of the \textit{restarted-conservative} algorithm that has been heuristically outlined in \cite{ScCF22}, and in Section~\ref{sec:acc_l1} we show that the per-iteration decay achieved by Algorithm~\ref{Alg_acc_subgr_l1} is always larger or equal than in Algorithm~\ref{Alg_subgr_l1}. \\
Finally, in Section~\ref{sec:num_exp} we test our algorithms in strongly and non-strongly convex optimization problems with $\ell_1$-regularization.

\section{Preliminary results} \label{sec:prel_res}
In this section we establish some auxiliary results that will be used later.
Given a convex function $f:\R^n\to\R$, for every $x\in\R^n$ we denote with $\partial f(x)\subset\R^n$ the subdifferential of $f$ at the point $x$. We recall that
\begin{equation*}
\partial f(x):=\{
y\in\R^n\mid f(z) \geq f(x) 
+ \langle y, z-x\rangle, \,\,
\forall z \in \R^n
\}.
\end{equation*}

\begin{defn} \label{def:min_norm_subd}
Let $f:\R^n\to\R$ be a convex function. 
For every  $x\in \R^n$, we define the vector
$\partial^-f(x) \in \R^n$ as follows 
\begin{equation} \label{eq:def_min_norm_subd}
\partial^-f(x) := \arg \min\{ |y|_2 \mid y\in \partial f(x) \}.
\end{equation}
\end{defn}

\begin{remark}
We observe that Definition~\ref{def:min_norm_subd} is always well-posed. Indeed, for every convex function $f:\R^n\to\R$, for every $x\in \R^n$ the subdifferential $\partial f(x)$ is a non-empty, compact and convex subset of $\R^n$. Namely, since we do not allow $f$ to assume the value $+\infty$, this fact descends directly from \cite[Theorem~3.1.15]{N18}. Moreover, we can equivalently rephrase \eqref{eq:def_min_norm_subd} as
\begin{equation*}
\partial^-f(x) := \arg \min\{ |y|_2^2 \mid y\in \partial f(x) \},
\end{equation*}
i.e., as a positive-definite quadratic programming problem on a convex domain. Hence, we deduce that $\partial^-f(x)$ is well-defined, and that it consists of a single element. Considering this last fact, in this paper we understand $\partial^- f:\R^n\to\R^n$ as a vector-valued operator, rather than a set-valued mapping.
\end{remark}

We report below a non-smooth version of the celebrated Polyak-$\L$ojasiewicz inequality.
We refer the reader to \cite{P63} and \cite[Theorem~2.1.10]{N18} for the classical statement in the smooth case, and to \cite[Section~2.3]{BNP17} and \cite[Section~2.2]{XY13} for the extension to non-differentiable functions.

\begin{lemma} \label{lem:non_smooth_PL}
Let $f:\R^n\to\R$ be a $\mu$-strongly convex function,
and let $x^*$ be its minimizer.
Then, for every $x\in \R^n$ and for every
element of the subdifferential $y\in \partial f(x)$
the following inequality holds:
\begin{equation*}
f(x)-f(x^*) \leq \frac1{2\mu} |y|_2^2,
\end{equation*}
and, in particular, we have
\begin{equation} \label{eq:Pol_subd_min}
f(x)-f(x^*) \leq \frac1{2\mu} |\partial^-f(x)|_2^2.
\end{equation}
\end{lemma}

\begin{proof}
Let us introduce the auxiliary 
function $\psi:\R^n\to\R$ defined as
\begin{equation*}
\psi(x):= f(x)- \frac\mu2 |x-x^*|^2_2.
\end{equation*}
The fact that $f$ is $\mu$-strongly convex guarantees that $\psi$ is still a convex function.
Moreover, for every $x\in \R^n$ we have that
\begin{equation} \label{eq:subd_diff}
\partial \psi(x) = \partial f(x) - \mu(x-x^*).
\end{equation} 
This follows immediately from the fact that
$f(x) = \psi(x) +\frac\mu2 |x -x^*|_2^2$ for every $x\in \R^n$, and from the 
\textit{sum rule for subdifferentials} 
(see, e.g., \cite[Theorem~23.8]{R97}), i.e., 
$\partial f(x) = \partial \psi(x) + \mu(x-x^*)$.
For every $x\in \R^n$ and for every
$y\in \partial f(x)$ we compute
\begin{equation} \label{eq:comput_loj_ineq}
\begin{split}
\psi(x^*) &\geq  \psi(x)  +
\langle y -\mu(x-x^*), x^*-x \rangle
  = f(x) +\frac\mu2 |x-x^*|^2_2 + \langle y , x^*-x \rangle\\
 & \geq f(x) -\frac{1}{2\mu}|y|_2^2,
\end{split}
\end{equation} 
where we used \eqref{eq:subd_diff} and the subdifferential inequality for the convex function $\psi$.
Recalling that $\psi(x^*) = f(x^*)$, 
from \eqref{eq:comput_loj_ineq} we directly deduce
the thesis. 
\end{proof}

We now introduce the class of functions that will be the main object of our investigation. We consider a \textit{composite} objective (see \cite{N13}) $f:\R^n\to \R$ of the form
\begin{equation}\label{eq:compos_l1}
f(x) = g(x) + \gamma |x|_1,
\end{equation}
where $g:\R^n\to\R$ is a $C^1$-regular convex function with Lipschitz-continuous gradient of constant $L>0$, and where $\gamma>0$ is a positive constant. We recall that $|x|_1:=\sum_{i=1}^n|x_i|$ for every $x\in\R^n$.
We observe that
\begin{equation} \label{eq:subd_comput_l1}
\partial f(x) = \left\{ 
\nabla g(x) +\gamma \sum_{i=1}^n\nu_i e_i\mid 
\nu_i = \sign (x_i) \mbox{ if } x_i\neq0,\,\,
\nu_i\in [-1,1] \mbox{ if } x_i=0
\right\}
\end{equation}
for every $x\in \R^n$, where $e_i$ is the $i$-th element of the standard basis of $\R^n$. If we define $\partial_i f(x):= \langle e_i, \partial f(x)\rangle$, we have that
\begin{equation}\label{eq:subd_compon_l1}
\partial_i f(x) =
\begin{cases}
\{ \partial_i g(x) + \gamma \nu_i\mid \nu_i=\sign(x_i) \} & x_i\neq 0,\\
\{ \partial_i g(x) + \gamma \nu_i\mid \nu_i\in [-1,1] \} & x_i=0,
\end{cases}
\end{equation}
for every $i=1,\ldots,n$, where $\partial_i g(x) := \frac{\partial}{\partial x_i}
g(x)$ denotes the usual partial derivative of the regular term $g:\R^n\to\R$ at the right-hand side of \eqref{eq:compos_l1}.
From \eqref{eq:subd_compon_l1} we read that the $i$-th component of $\partial f(x)$ is affected only by $\nu_i$. Therefore, in order to compute the operator $\partial^-f:\R^n\to\R^n$ introduced in Definition~\ref{def:min_norm_subd}, we can find \textit{separately} the element of minimal absolute value of $\partial_i f(x)$ for $i=1,\ldots, n$.
We use $\partial_i^-f(x)$ to access the $i$-th component of $\partial^-f(x)$. In particular, for every $x\in\R^n$ we have that
\begin{equation}\label{eq:expr_sub-_f}
\partial_i^- f(x) = 
\begin{cases}
\partial_i g(x) + \gamma \sign (x_i)
& x_i\neq 0,\\
\sign (\partial_i g(x))
\max\{ |\partial_i g(x)| -\gamma,0 \}
& x_i=0.
\end{cases}
\end{equation}

\begin{defn}\label{def:partition}
Given $x = (x_1,\ldots,x_n)\in \R^n$, we define the following partition of the components $\{ 1,\ldots,n \}$ induced by the point $x$: 
\begin{equation} \label{eq:def_alfa_beta}
\begin{split}
\alpha^+_x &:= \{ i\in\{1,\ldots,n\} \mid x_i>0 \},\\
\alpha^-_x &:= \{ i\in\{1,\ldots,n\} \mid x_i<0 \},\\
\beta_x &:= \{ i\in\{1,\ldots,n\} \mid x_i=0 \}.
\end{split}
\end{equation}
\end{defn}
From now on, when making use of a partition $\alpha^1,\ldots,\alpha^k$ of the indexes of the components $\{1,\ldots,n\}$, for every $z=(z_1,\ldots,z_n)\in \R^n$ we write $z=(z_{\alpha^1},\ldots,z_{\alpha^k})$, where $z_{\alpha^j}\in \R^{|\alpha^j|}$ is the vector obtained by extracting from $z$ the components that belong to $\alpha^j$, i.e., $z_{\alpha^j}=(z_i)_{i\in \alpha^j}$ for every $j=1,\ldots,k$.
The next technical result is the key-lemma of the convergence proof of Section~\ref{sec:subgr_meth}. 

\begin{lemma} \label{lem:direction_decrease}
Let $f:\R^n\to\R$ be a convex function of the 
form \eqref{eq:compos_l1}.
Given $x\in \R^n$, let $\alpha^+_x,\alpha^-_x,\beta_x$
be the partition of $\{ 1,\ldots,n\}$ 
corresponding to the point $x$ and
prescribed by \eqref{eq:def_alfa_beta}.
Let us consider a vector 
 $v=(v_1,\ldots,v_n)\in \R^n$
such that
\begin{equation} \label{eq:hp_v}
\begin{cases}
x_i + v_i \geq 0 &\forall i\in \alpha^+_x,\\
x_i + v_i \leq 0 &\forall i\in \alpha^-_x,\\
v_i = 0 & \forall i\in \beta_x \mbox{ s.t. }
\partial^-_if(x) =0,\\
v_i \geq 0 & \forall i\in \beta_x \mbox{ s.t. }
\partial^-_if(x) <0,\\
v_i \leq 0 & \forall i\in \beta_x \mbox{ s.t. }
\partial^-_i f(x) >0.
\end{cases}
\end{equation}
Then the following inequality holds: 
\begin{equation}\label{eq:dec_est}
f(x+v) \leq f(x) +\langle \partial^- f(x) , v \rangle
+ \frac12 L|v|_2^2,
\end{equation}
where $L>0$ is the Lipschitz constant of the regular term at the right-hand side of \eqref{eq:compos_l1}, and $\partial^-f(x)$ is defined as in Definition~\ref{def:min_norm_subd}.
\end{lemma}
\begin{remark}
We recall that, in the case of a regular convex function $\phi:\R^n\to\R$ with $L$-Lipschitz continuous gradient, we have \begin{equation} \label{eq:dec_est_reg}
\phi(x+v) \leq \phi(x) +\langle\nabla \phi(x), v\rangle + \frac12 L|v|_2^2
\end{equation}
for every $x,v\in \R^n$ (see, e.g., \cite[Theorem~2.1.5]{N18}). The crucial fact for the proof of Lemma~\ref{lem:direction_decrease} is that, when $v$ satisfies the conditions \eqref{eq:hp_v}, the segment $\overrightarrow{xx'}$ lies in a region where the restriction of the objective $f$ is regular, where we set $x':=x+v$. 
Lemma~\ref{lem:direction_decrease} will be used to prove that, along proper directions, the objective function $f$ is decreasing.
\end{remark}

\begin{proof}
Before proceeding, we introduce another partition of the set of indexes $\beta_x$:
\begin{align*}
\beta^+_x &:=\{ i\in\beta_x\mid v_i>0 \},\\
\beta^-_x &:=\{ i\in\beta_x\mid v_i<0 \},\\
\beta^0_x &:=\{ i\in\beta_x\mid v_i=0 \},
\end{align*} 
and we define
\begin{equation*}
\zeta^+ := \alpha^+_x \cup \beta^+_x, \qquad
\zeta^- := \alpha^-_x \cup \beta^-_x, \qquad
\zeta^0 :=  \beta^0_x
\end{equation*}
where $\alpha^+_x, \alpha^-_x, \beta_x$ are set accordingly to \eqref{eq:def_alfa_beta}.
If we consider the segment $t\mapsto \eta(t) = x+tv$ for $t\in[0,1]$, it turns out that 
$\eta(t)\in C_{\zeta^{\pm,0}}$ for every $t\in[0,1]$,
where
\begin{equation*}
C_{\zeta^{\pm,0}} := \left\{ z\in \R^n\mid 
z_{\zeta^+}\geq 0, z_{\zeta^-}\leq 0, z_{\zeta^0}= 0 \right\}.
\end{equation*}
Let us define the auxiliary function $\faux:\R^{n}\to\R$ as 
\begin{equation} \label{eq:def_f_aux}
\faux: z=
(z_{\zeta^+}, z_{\zeta^-},z_{\zeta^0})
\mapsto
g(z_{\zeta^+},z_{\zeta^-},0_{\zeta^0})+
\gamma \sum_{i  \in \zeta^+}z_i
- \gamma \sum_{i \in \zeta^-}z_i,
\end{equation}
where $g:\R^n\to\R$ is the smooth term at the right-hand side of \eqref{eq:compos_l1}.
From the definition of $\faux$, it follows that
\begin{equation} \label{eq:grad_faux}
\nabla \faux: z=
(z_{\zeta^+}, z_{\zeta^-},z_{\zeta^0})
\mapsto
\nabla g(z_{\zeta^+},z_{\zeta^-},0_{\zeta^0})+
\gamma \sum_{i  \in \zeta^+}e_i
- \gamma \sum_{i \in \zeta^-}e_i.
\end{equation}
We observe that the function $\faux:\R^n\to\R$ is as regular as $g$, i.e., it is of class $C^1$ with $L$-Lipschitz continuous gradient. Indeed, the first term at the right hand-side of \eqref{eq:grad_faux} is obtained as the composition $\nabla g \circ \Pi_{\zeta^0}$, where $\Pi_{\zeta^0}:\R^n\to\R^n$ is the linear ($1$-Lipschitz) orthogonal projection onto the subspace $\{ z\in \R^n\mid z_{\zeta^0}=0 \}\subset\R^n$. Moreover, the last terms at the right hand-side of \eqref{eq:grad_faux} are constant.
Therefore, using the identity
\[
f|_{C_{\zeta^{\pm,0}}} \equiv \faux
|_{C_{\zeta^{\pm,0}}},
\]
if we apply the estimate \eqref{eq:dec_est_reg} to $\faux$, we deduce that
\begin{equation*}
f(x+v) \leq f(x) + \langle 
\nabla \faux (x), v \rangle + \frac12L|v|_2^2.
\end{equation*}
Therefore, the thesis follows if we show that the following equalities hold:
\begin{equation} \label{eq:last_step_lemma}
 \frac{\partial}{\partial x_i} \faux (x) v_i 
=
 \partial_i^- f (x) v_i \qquad i=1,\ldots,n.
\end{equation}
Using the partition of the components
$\{1,\ldots,n\}$ provided by 
the families of indexes $\alpha^+_x$, $\alpha^-_x$, $\beta^+_x$, $\beta^-_x$, and $\beta^0_x$, we have the following possibilities:
\begin{itemize}
\item If $i \in \alpha^+_x$, in virtue of
\eqref{eq:expr_sub-_f} and \eqref{eq:def_f_aux}, we obtain
$\partial_i^- f(x) = \frac{\partial}{\partial x_i} g(x) + \gamma=\frac{\partial}{\partial x_i}\faux(x)$.
\item The case $i\in \alpha^-_x$ is analogous to
$i\in\alpha^+_x$.
\item If $i \in \beta^+_x$, then $x_i=0$ and $v_i>0$, and, in virtue
of \eqref{eq:hp_v}, we deduce that 
$\partial_i^-f(x) <0$. In particular, using again 
\eqref{eq:expr_sub-_f}, this implies that 
$\partial_i^-f(x) = \frac{\partial}{\partial x_i}g(x)+\gamma$.
 On the other
hand, recalling the expression of $\faux$ in \eqref{eq:def_f_aux} and the inclusion $\beta^+_x\subset \zeta^+$,
we finally deduce $\frac{\partial}{\partial x_i} \faux (x) = \partial_ig(x) +\gamma$.
\item The case $i\in \beta^-_x$ is analogous to
$i\in\beta^+_x$.
\item If $i\in \beta^0_x$, then $v_i=0$, and we immediately obtain $\partial_i^-f(x)v_i = 0 =\frac{\partial}{\partial x_i}\faux(x)v_i$.
\end{itemize}
This argument shows that \eqref{eq:last_step_lemma}
is true, and it concludes the proof. 
\end{proof}

\section{Subgradient method and convergence analysis} \label{sec:subgr_meth}
In this section we propose a subgradient method with constant step-size for the numerical minimization of a convex function $f:\R^n\to\R$ with the \textit{composite} structure reported in \eqref{eq:compos_l1}. We insist on the fact that the analysis presented here holds only when the non-smooth term at the right-hand side of \eqref{eq:compos_l1} is a $\ell_1$-penalization.\\
Before introducing formally the algorithm, we provide some insights that have guided us towards its construction.
Let $\bar x\in \R^n$ be the current guess for the minimizer of $f$. We want to find a suitable direction in the subdifferential $\subel\in \partial f(\bar x)$ such that $f(\bar x-h\subel)\leq f(\bar x)$, where $h>0$ represents a \textit{constant} step-size. In order to accomplish this, a natural choice consists in setting $\subel = \partial^-f(\bar x)$, where $\partial^-f(\bar x)$ is defined as in \eqref{eq:def_min_norm_subd}.
To see this, we first observe that, in virtue of the particular structure of $\partial f$ reported in \eqref{eq:subd_compon_l1}, we can choose \textit{separately} the components $\subel_1,\ldots, \subel_n$ of the direction of the movement. If $\bar x_i\neq 0$, then $\partial_i f(\bar x)$ consists of a single element, hence the only possible choice is $\subel_i = \partial^-_i f(\bar x)$. 
If $\bar x_i = 0$ and $\partial^-_i f(\bar x) =0$, then the convex application $t\mapsto f(\bar x + t e_i)$ attains the minimum at $t=0$. Hence, any choice $\subel_i\in \partial_i f(\bar x)$ with $\subel_i\neq 0$ would give $f(\bar x - h\subel_i e_i)\geq f(\bar x)$, resulting in an increase of the objective function. 
For this reason, it is convenient to set $\subel_i=\partial^-_i f(\bar x)=0$, and to move \textit{tangentially} to the non-differentiability region $\{ x\in \R^n\mid x_i=0 \}$.
On the other hand, if $\bar x_i = 0$ and, e.g., $\partial^-_i f(\bar x) >0$, then $\partial_i f(\bar x)\subset (0,+\infty)$, and for every choice of $\subel_i\in \partial_i f(\bar x)$, we have that $\bar x_i - h\subel_i = - h\subel_i <0$. However, observing that $\lim_{h\to 0^+} (f(\bar x+he_i)-f(\bar x))/h=\partial^-_i f(\bar x)$, it looks natural to set once again $\subel_i = \partial^-_i f(\bar x)$.\\
Besides the selection of the direction $\subel =\partial^-f(\bar x)$, the second crucial aspect is whether some sign changes occur in the coordinates when moving from $\bar x$ to $\bar x - h\partial^- f(\bar x)$. 
If not, the situation is pretty analogous to a step of the classical gradient descent in the smooth framework. On the other hand, if there is, e.g., a positive component $\bar x_i$ that becomes negative, then we should carefully decide if the \textit{barrier} $\{x\in \R^n\mid x_i=0  \}$ should be crossed, or not. 
This is a key-point, in order to avoid the oscillations that characterized the simple example in the Introduction. In this case, we first set to $0$ the components involved in a sign change, and for these components we re-evaluate $\partial^- f$. 
Finally, using this additional information, we complete the step, as depicted in Figure~\ref{fig:cross_l1}. The implementation of the method is described in Algorithm~\ref{Alg_subgr_l1}.

\begin{figure}
\centering
\includegraphics[scale=0.2]{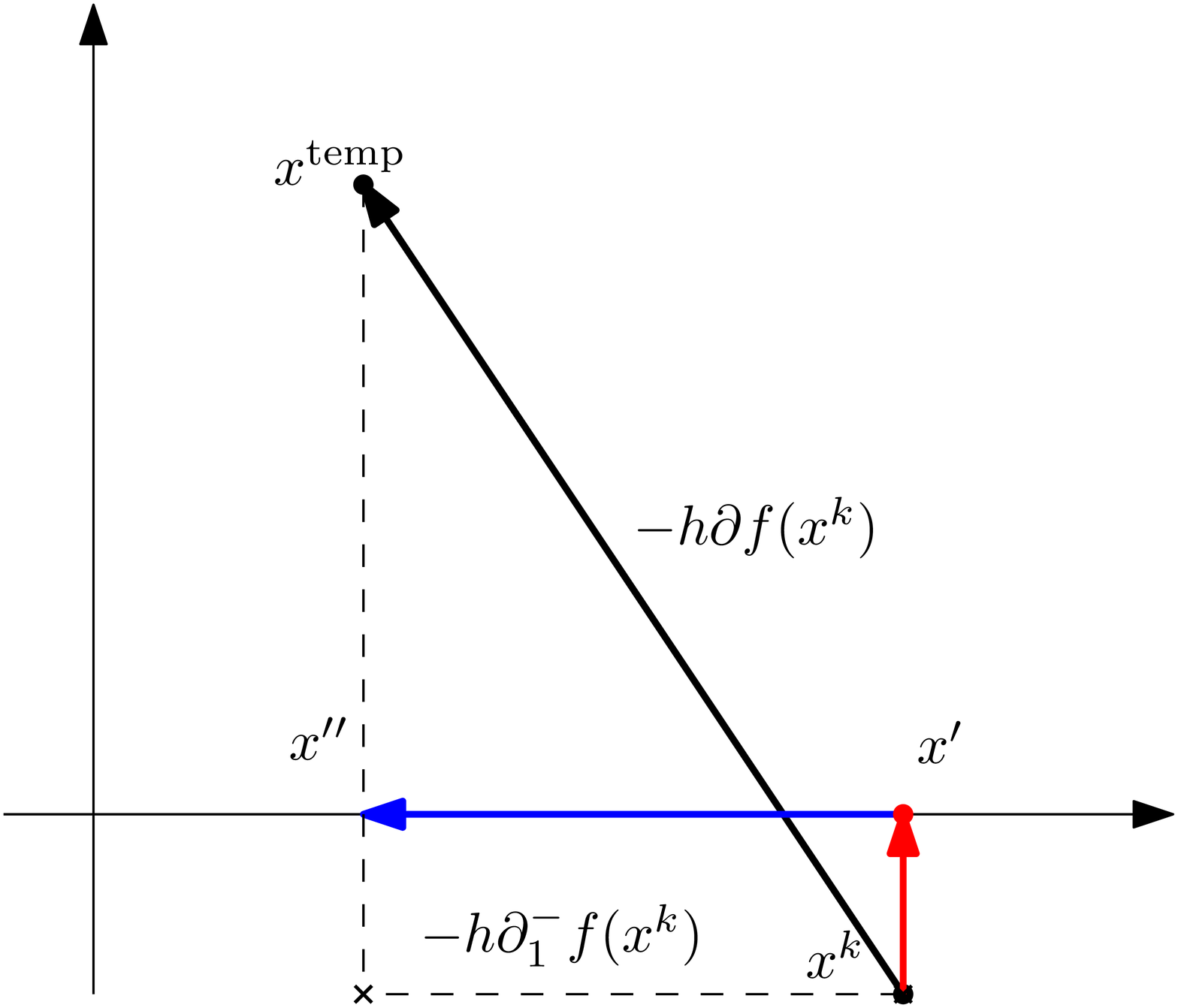}
\includegraphics[scale=0.2]{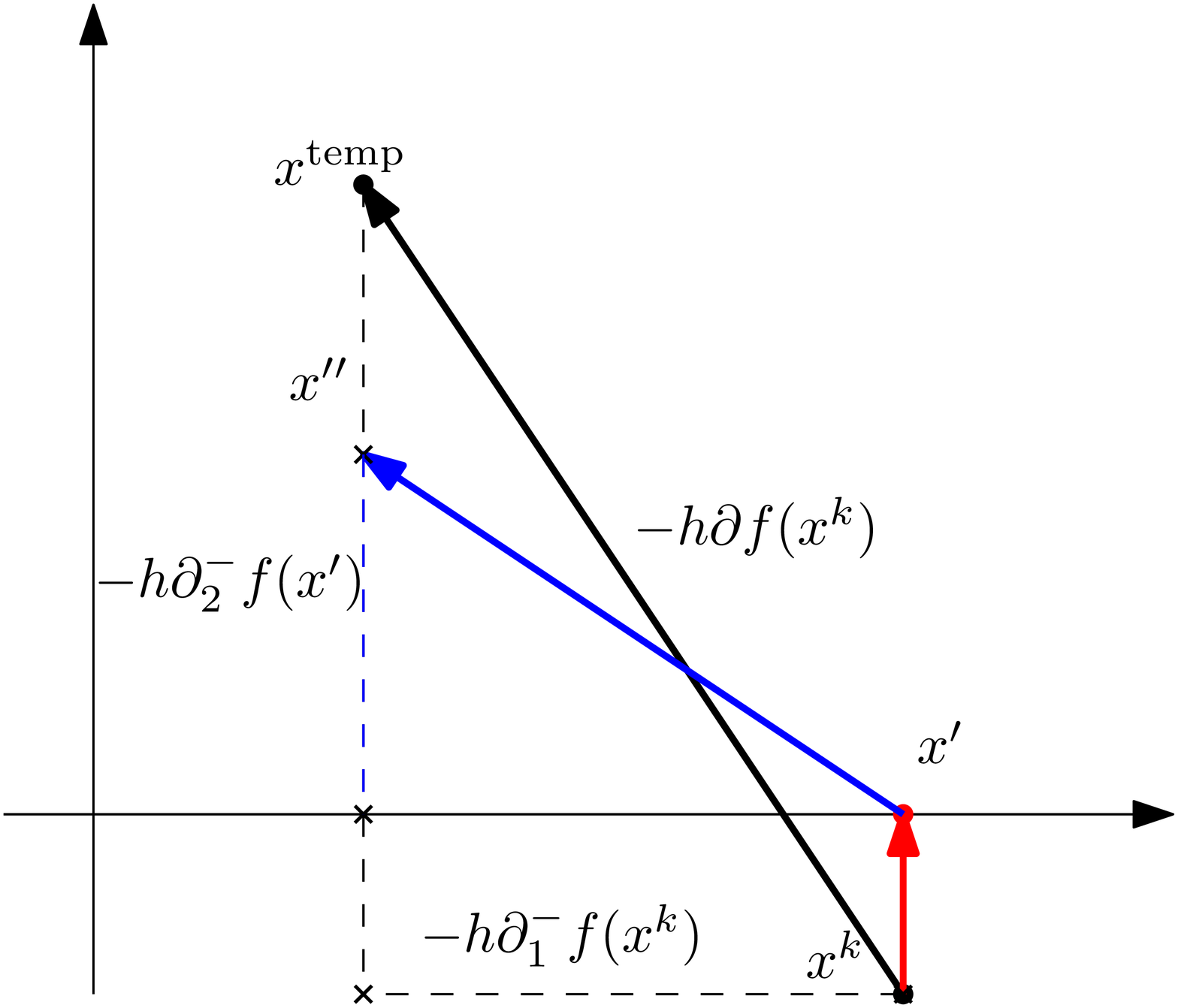}
\caption{In this $2D$-example, we observe that the second component $x_2$ changes sign when moving from $x^k$ to $\xtemp = x^k-h\partial^-f(x^k)$. Hence, we first consider $x'=(x^k_1,0)$ and we evaluate $\partial^-_2 f(x')$. If $\partial^-_2 f(x')=0$ (see the picture on the left), then we complete the step by moving tangentially to the axis $\{ x_2=0 \}$, in the direction $-(\partial^-_1 f(x^k),0)$. If $\partial^-_2 f(x')\neq 0$ (right), then we complete the step using the direction $-(\partial^-_1 f(x^k),\partial^-_2 f(x'))$.}
\label{fig:cross_l1}
\end{figure}

\begin{algorithm}
\caption{Subgradient method for $\ell_1$-composite optimization}
\begin{algorithmic}[1]
  \scriptsize
  \STATE $x^0$ initial guess, $h>0$ constant step-size
  \STATE $k \gets 0$
  \WHILE{$k \leq \max_{\mathrm{iter}}$}
  \STATE $x^{\mathrm{temp}} \gets x^k - h\partial^- f(x^k)$
  \IF {$\sign(x_i^{\mathrm{temp}}\cdot x_i^k )\geq 0, \forall i =1,\ldots,n$}
  \STATE $x^{k+1}\gets x^{\mathrm{temp}}$
  \ELSE
  \STATE $I\gets \{ i \mid \sign(x_i^{\mathrm{temp}}\cdot x_i^k )\leq 0 \}$
  \STATE $x'_i\gets x_i^k, \quad \forall i\not\in I$
  \STATE $x'_i\gets 0, \quad \forall i\in I$
  \STATE $v_i'' \gets -h\partial_i^-f(x^k), \quad \forall i\not\in I$
  \STATE $v_i'' \gets -h\partial_i^-f(x'), \quad \forall i\in I$
  \STATE $x''\gets x' + v''$
  \IF{$f(x')<f(x'')$}
  \STATE $x^{k+1}\gets x'$
  \ELSE
  \STATE $x^{k+1}\gets x''$
  \ENDIF
  \ENDIF
  \STATE $k \gets k+1$
  \ENDWHILE
\end{algorithmic}
\label{Alg_subgr_l1}
\end{algorithm}

We now establish the linear convergence result for Algorithm~\ref{Alg_subgr_l1} in the case of strongly convex objective.

\begin{theorem}\label{thm:lin_conv_subg}
Let $f:\R^n\to\R$ be a function such that $f(x) =g(x) +\gamma |x|_1$ for every $x\in \R^n$, where $\gamma>0$ and $g:\R^n\to\R$ is  $C^1$-regular. We further assume that there exist constants $L>\mu>0$ such that $g$ is $\mu$-strongly convex and $\nabla g$ is $L$-Lipschitz continuous. Let $(x^k)_{k\geq 0}$ be the sequence generated by Algorithm~\ref{Alg_subgr_l1}. Then, there exists $\kappa = \kappa(L,\mu)\in (0,1)$ such that
\begin{equation}\label{eq:lin_conv_subg}
f(x^k)-f(x^*) \leq \kappa^k (f(x^0)-f(x^*)),
\end{equation}
where $x^*\in\R^n$ denotes the unique minimizer of $f$, and where we set the step-size $h=\frac1L$.
\end{theorem}

\begin{proof}
We follow the procedure described in Algorithm~\ref{Alg_subgr_l1}. We prove that each iteration leads to a linear decrease of the value of the objective function.
The first stage of each step is based on the following update:
\begin{equation} \label{eq:update_provv}
\xtemp := x^k - h\partial^-f(x^k),
\end{equation}
where $h>0$ represents the step-size of the
sub-gradient method.
We distinguish two possible scenarios, corresponding to the \texttt{if-else} statement at the lines~5 and~7 of Algorithm~\ref{Alg_subgr_l1}.\\
\textbf{Case 1.} We have that 
\begin{equation} \label{eq:subgr_case_1}
\sign(\xtemp_i \cdot x^k_i)\geq 0, \quad 
\forall i=1,\ldots,n, 
\end{equation}
i.e., none of the components of $ x^k $ and of $\xtemp$ changes sign, in the sense that from strictly positive it becomes strictly negative, or vice-versa.
If we set $v:=-h\partial^-f(x^k)$, we observe that the hypotheses of Lemma~\ref{lem:direction_decrease} are met for the point $x^k$ and the vector $v$. Indeed, using the partition introduced in \eqref{eq:def_alfa_beta} and induced by the point $x^k$, from \eqref{eq:subgr_case_1} it follows that $i\in \alpha^+_{x^k}$ implies $x^k_i+v_i\geq 0$. A similar argument holds for $i\in \alpha^-_{x^k}$. Finally, if $i\in \beta_{x^k}$, then $v_i$ satisfies \eqref{eq:hp_v} by construction.
Therefore, from \eqref{eq:dec_est} we deduce that 
\begin{equation*} 
\begin{split}
f(\xtemp) &\leq f(x^k) + \langle \partial^-f(x^k)
 , v \rangle + \frac12 L |v|_2^2 \\
& \quad = f(x^k) - \left( h - \frac{h^2}{2}L \right)
\left| \partial^- f(x^k) \right|_2^2.
\end{split} 
\end{equation*}
Moreover, if $h\leq \frac2L$, in virtue of Lemma~\ref{lem:non_smooth_PL},
we obtain that
\begin{equation*} 
f(\xtemp) -f(x^*) \leq \left( f(x^k) - f(x^*) \right)
\left(1 -  2\mu\left( h - \frac{h^2}{2}L \right) 
\right).
\end{equation*}
In this case, we assign $x^{k+1}:=\xtemp$ and, choosing $h=\frac1L$ in order to minimize the right-hand side of the previous inequality, we get
\begin{equation}\label{eq:dec_case_1_fin}
f(x^{k+1}) - f(x^*) \leq \left( 1- \frac{\mu}L \right)
\left( f(x^k) - f(x^*) \right).
\end{equation}
\textbf{Case 2.} Recalling the definition of $\xtemp$ in \eqref{eq:update_provv}, we are in the second scenario when
\begin{equation}\label{eq:subgr_case_2}
\exists i\in \{ 1,\ldots,n \}:\, 
(x_i^k>0 \mbox{ and } \xtemp_i< 0)
\mbox{ or } (x_i^k< 0 \mbox{ and } \xtemp_i> 0),
\end{equation}
i.e., there is at least one component that \textit{strictly} changes sign.
Before proceeding, we introduce 
the following partition of the components:
\begin{equation} \label{eq:divis_comp_2}
\begin{cases}
i \in \xi^+_{x^k} & 
\mbox{if } (x^k_i>0 \mbox{ and } 
\xtemp_i> 0), \\
i \in \xi^-_{x^k} & 
\mbox{if } (x^k_i<0 \mbox{ and } \xtemp_i< 0),\\
i \in \xi^0_{x^k}  & 
\mbox{if }\,
\sign(\xtemp_i \cdot x_i^k)\leq 0,
\end{cases}
\end{equation}
and we define the following intermediate
points:
 \begin{align}\label{eq:def_x'}
{\mbox{\bf \textit{Phase (1)} }} \qquad  x' := 
(x^k_{\xi^+_{x^k}},x^k_{\xi^-_{x^k}},0_{\xi^0_{x^k}} ),
\end{align}
and
\begin{align} \label{eq:def_x''}
{\mbox {\bf \textit{Phase (2)} }} \qquad   x'' := x' +  v'',
\end{align}
where
\begin{equation} \label{eq:def_v''_inex_subgr}
v'' := -h \left(\partial^-_{\xi^+_{x^k}}f(x^k),
\partial^-_{\xi^-_{x^k}}f(x^k), 
\partial^-_{\xi^0_{x^k}}f(x')\right).
\end{equation}
We observe that \eqref{eq:def_x'} corresponds to the assignments of lines~9-10 in Algorithm~\ref{Alg_subgr_l1}, while \eqref{eq:def_v''_inex_subgr} incorporates lines~11-12. Finally, $x''$ is defined in \eqref{eq:def_x''} accordingly to line~13.
We insist on the fact that in the update 
\eqref{eq:def_x''} the vector $v''$ is computed 
by re-evaluating $\partial^-_{\xi^0_{x^k}}f$  
at the point $x'$. This is because 
$\partial^-_{\xi^0_{x^k}}f$ may exhibit sudden changes
when considering the points $x^k$ and $x'$. 
In this regard, our
construction guarantees that we employ the
most trustworthy values for the choice of the
decrease direction $v''$. We point out that, if $x^k_i=0$, then $i\in \xi^0_{x^k}$. Moreover, we remark that if $i\in \xi^0_{x^k}$ and $\partial_i f(x^k)=0$, then we have necessarily that $x^k_i=0$. Indeed, in this case, from \eqref{eq:update_provv} and $\partial_i f(x^k)=0$ it follows that $\xtemp_i = x_i^k$, while $i\in \xi^0_{x^k}$ gives $\sign( x_i^k \cdot  x_i^k)\leq 0$, resulting in $x_i^k=0$.\\
{\bf \textit{Phase (1).}} From \eqref{eq:def_x'}, we immediately observe that 
\begin{equation*}
x' = x^k + v',
\end{equation*}
with 
\begin{equation} \label{eq:def_v'_step_1}
v_i' :=
\begin{cases}
0 & \mbox{if } i \in \xi^+_{x^k} \cup \xi^-_{x^k},\\
 -   \eta_i  h\partial^-_i f(x^k)
& \mbox{if } i\in \xi^0_{x^k},
\end{cases}
\end{equation}
and where, for every $i\in\xi^0_{x^k}$, we set
\begin{equation*}
\eta_i := 
\begin{cases}
\frac{x_i^k}{h\partial_i^- f(x^k)} & \mbox{if } \partial_i^- f(x^k) \neq 0,\\
0 & \mbox{if } \partial_i^- f(x^k) = 0.
\end{cases}
\end{equation*}
We first notice that $\eta_i \in [0,1]$. Indeed, 
assuming that $\partial_i^- f(x^k) \neq 0$ (otherwise
there is nothing to prove), since
$i\in\xi^0_{x^k}$, recalling \eqref{eq:divis_comp_2} and \eqref{eq:update_provv}, we have
\begin{equation} \label{eq:comp_sign_eta}
x_i^k\left( x_i^k -
h\partial^-_if(x^k) \right)\leq 0,
\end{equation}
 which in turn
gives $x_i^k \partial^-_if(x^k)\geq 0$ and,
as a matter of fact, $\eta_i\geq 0$.
On the other hand, in order to show that $\eta_i\leq1$,
we assume without loss of generality 
that $x^k_i\neq 0$. Then, 
using again \eqref{eq:comp_sign_eta}, 
it follows that 
\begin{equation*}
0\geq 
\left( 1 - \frac{h\partial^-_if(x^k)}{x_i^k} \right)
= \left( 1- \frac1\eta_i \right),
\end{equation*}
that yields $\eta_i\leq1$.
Therefore, we conclude that
\begin{equation}\label{eq:bound_eta_i}
0\leq \eta_i\leq 1 \qquad \forall i \in \xi^0_{x^k}.
\end{equation}
Finally, from \eqref{eq:subgr_case_2}
we deduce that there exists at least one 
index $\hat i \in \xi^0_{x^k}$ such that
$\eta_{\hat i}>0$.\\
Using the partition $\alpha^+_{x^k},\alpha^-_{x^k},\beta_{x^k}$
of $\{ 1, \ldots, n\}$ induced by the point $x^k$
and prescribed by \eqref{eq:def_alfa_beta},
 we obtain that
the following conditions are satisfied:
\begin{itemize}
\item If $i\in \alpha^+_{x^k}$, then 
either $i  \in \xi^+_{x^k}$ or 
$i  \in \xi^0_{x^k}$. In the first case, $x'_i=x^k_i>0$, then $v'_i=0$. In the second, $x'_i = 0 = x^k_i + v'_i$. Hence, in any case, $x^k_i + v'_i\geq 0$. 
\item If $i\in \alpha^-_{x^k}$, then an analogous reasoning as before yields $x^k_i + v'_i\leq 0$. 
\item If $i\in \beta_{x^k}$, then $x^k_i=0$. Hence, $i\in \xi^0_{x^k}$, and $x'_i =0$. Therefore, $v'_i=0$.
\end{itemize}
The previous argument proves that the vector $v'$ introduced in \eqref{eq:def_v'_step_1} satisfies the assumptions of Lemma~\ref{lem:direction_decrease} at the point $x^k$.
Thus, we deduce that
\begin{equation} \label{eq:dec_case_2_1}
\begin{split}
f(x') &\leq  
f(x^k) + \langle \partial^-f(x^k), v' \rangle
+ \frac{L}2|v'|^2_2 \\
& \quad =  f(x^k) - \sum_{i\in \xi^0_{x^k}} 
\left(
h\eta_i - h^2\eta_i^2\frac{L}{2}
\right)|\partial_i^-f(x^k)|^2.
\end{split}
\end{equation}
If we set $\bar \eta := \max\{\eta_i\mid i\in \xi^0_{x^k}\}$,
we observe that \eqref{eq:dec_case_2_1} 
implies that $f(x')\leq f(x^k)$ whenever 
$h\in \left[0, \frac{2}{L\bar\eta} \right]$. 
We stress the fact that the condition
 \eqref{eq:subgr_case_2}
that characterizes the present scenario
guarantees that $\bar \eta>0$.\\
{\bf \textit{Phase (2).}}
We now investigate the update described in
\eqref{eq:def_x''}-\eqref{eq:def_v''_inex_subgr}.
Let $\alpha^+_{x'},\alpha^-_{x'}$ and $\beta_{x'}$ be the partition of the components $\{ 1,\ldots,n \}$ induced by the point $x'$ and prescribed by \eqref{eq:def_alfa_beta}.
Recalling \eqref{eq:divis_comp_2} and the 
definition of $x'$ in \eqref{eq:def_x'}, we 
observe that $\alpha^+_{x'} = \xi^+_{x^k}$, $\alpha^-_{x'} = \xi^-_{x^k}$
and $\beta_{x'} = \xi^0_{x^k}$.
Hence, since $x'_i=x^k_i\neq0$ for every $i\in \alpha^+_{x'}\cup \alpha^-_{x'}$, from \eqref{eq:expr_sub-_f} it descends that
that
\begin{equation*} 
\partial^-_if(x')-\partial^-_if(x^k)
=  \partial_i g(x')-\partial_i g(x^k) \qquad
\forall i\in\alpha^+_{x'}\cup\alpha^-_{x'},
\end{equation*}
which, in virtue of \eqref{eq:def_v''_inex_subgr}, yields 
\begin{equation} \label{eq:sum_v''_subgrad_x'}
v''_i + h \partial^-_if(x')
=  h\left(
\partial_i g(x') - \partial_i g(x^k)
\right) \qquad
\forall i\in\alpha^+_{x'}\cup\alpha^-_{x'}.
\end{equation}
Moreover, using \eqref{eq:def_v''_inex_subgr}, \eqref{eq:update_provv} and \eqref{eq:divis_comp_2}, we deduce
that
\begin{equation}\label{eq:hyp_lem_v''_alpha}
\begin{cases}
i\in \alpha^+_{x'}  \implies
x''_i = x^k_i - h\partial^-_if(x^k)>0 \implies x''_i = x'_i + v''_i >0,\\
i\in \alpha^-_{x'}  \implies x''_i = x^k_i - h\partial^-_if(x^k)<0 \implies x''_i = x'_i + v''_i <0.
\end{cases}
\end{equation}
On the other hand, from \eqref{eq:def_v''_inex_subgr}
and recalling that $\beta_{x'} = \xi^0_{x^k}$,
we have that
\begin{equation}\label{eq:hyp_lem_v''_beta}
i\in \beta_{x'} \implies v''_i = -h \partial_i^- f(x').
\end{equation}
By combining \eqref{eq:hyp_lem_v''_alpha} and
\eqref{eq:hyp_lem_v''_beta}, we obtain that the
hypotheses of Lemma~\ref{lem:direction_decrease}
are met when considering the point $x'$ and the
direction $v''$. Hence, it follows that
\begin{equation} \label{eq:dec_case_2_2}
\begin{split} 
f(x'')  &\leq f(x') 
+ \langle \partial^- f(x'),
v''\rangle + \frac{L}{2}|v''|_2^2\\
&\quad = f(x') - 
\left( h-\frac{L}2h^2 \right)|\partial^- f(x')|_2^2\\
& \qquad
+ (1-Lh) \langle \partial^- f(x'), v'' + h\partial^- f(x') \rangle
+\frac{L}2 |v''+h\partial^-f(x')|_2^2.
\end{split}
\end{equation}
On the other hand, recalling \eqref{eq:def_v''_inex_subgr} and \eqref{eq:sum_v''_subgrad_x'}, we have that
\begin{equation} \label{eq:inex_term_2}
\begin{split}
|v''+h\partial^-f(x')|_2^2 &= h^2\sum_{i\in\alpha^+_{x'}\cup\alpha^-_{x'}} |\partial_ig(x^k)-\partial_ig(x')|^2 \leq h^2 |\nabla g(x^k) - \nabla g(x')|_2^2\\
& \leq h^4 L^2 \sum_{i\in \xi^0_{x^k}}
 \eta_i^2 |\partial^-_if(x^k)|^2,
\end{split}
\end{equation}
where we used the Lipschitz-continuity of $\nabla g$, \eqref{eq:def_v'_step_1} and the fact that $x'-x^k = v'$.
If we set $h=\frac1L$ in \eqref{eq:dec_case_2_2}, owing to \eqref{eq:inex_term_2} we deduce that
\begin{equation*}
f(x'') \leq f(x') - \frac{1}{2L}|\partial^- f(x')|_2^2
+ \frac1{2L} \sum_{i\in \xi^0_{x^k}}
 \eta_i^2 |\partial^-_if(x^k)|^2.
\end{equation*}
Moreover, by combining the last inequality with \eqref{eq:dec_case_2_1} (using again $h=\frac1L$), we obtain that
\begin{equation} \label{eq:dec_est_2_final}
\begin{split}
f(x'') &\leq f(x^k) - \frac{1}{2L}|\partial^- f(x')|_2^2
+ \frac1{L} \sum_{i\in \xi^0_{x^k}}
 (\eta_i^2-\eta_i) |\partial^-_if(x^k)|^2\\
&\leq  f(x^k) - \frac{1}{2L}|\partial^- f(x')|_2^2,
\end{split}
\end{equation}
where we used \eqref{eq:bound_eta_i} in the last passage. In virtue of Lemma~\ref{lem:non_smooth_PL}, 
from \eqref{eq:dec_est_2_final} we
deduce that
\begin{equation} \label{eq:dec_PL_case_2}
f(x'')-f(x^*) \leq f(x^k)- f(x^*)
- \frac{\mu}{L} \left( f(x') - f(x^*) \right).
\end{equation}
We now distinguish two possibilities, corresponding to the \texttt{if-else} statement at lines~14 and~16 of Algorithm~\ref{Alg_subgr_l1}.
\begin{itemize}
\item If $f(x')<f(x'')$, then we set 
$x^{k+1}:=x'$.
\item If $f(x'')\leq f(x')$, then we set 
$x^{k+1}:=x''$.
\end{itemize}
In any case, from \eqref{eq:dec_PL_case_2}
we obtain
\begin{equation} \label{eq:final_dec_case_2}
f(x^{k+1}) - f(x^*) \leq 
\left(
1 +\frac{\mu}{L}
\right)^{-1} (f(x^k)-f(x^*)).
\end{equation}
Finally, in virtue of \eqref{eq:dec_case_1_fin} and \eqref{eq:final_dec_case_2}, if we set 
\begin{equation*}
\kappa = \max\left(  
\left(
1 -\frac{\mu}{L}
\right),
\left(
1 +\frac{\mu}{L}
\right)^{-1}
\right),
\end{equation*}
we deduce the thesis.
\end{proof}

\begin{remark}
The hypothesis of the strong convexity of the smooth function $g:\R^n\to \R$ in Theorem~\ref{thm:lin_conv_subg} can be slightly relaxed by requiring that $g:\R^n\to \R$ is convex, that the objective $f:\R^n\to \R$ adimits a minimizer $x^*$ and that there exists a constant $\mu>0$ such that $f$ satisfies the inequality \eqref{eq:Pol_subd_min} for every $x\in \R^n$. 
Indeed, in the proof of Theorem~\ref{thm:lin_conv_subg} we only employ \eqref{eq:Pol_subd_min}, and we do not use the strong convexity assumption. On the other hand, the assumption of convexity for $g$ is needed for the notion of subgradient considered in this paper. 
\end{remark}

\section{Accelerated subgradient method}\label{sec:acc_l1}

In this section we propose a momentum-based acceleration of Algorithm~\ref{Alg_subgr_l1} for an objective function $f:\R^n\to\R$ with the $\ell_1$-composite structure introduced in \eqref{eq:compos_l1}. As observed in the Introduction, in the smooth-objective framework it is possible to design minimization schemes with momentum by discretizing second order ODEs of the form:
\begin{equation}\label{eq:diss_ODE}
\ddot x +\nabla V(x) = -A(x,t)\dot x \iff
\begin{cases}
\dot x = p\\
\dot p = -\nabla V(x) - A(x,t)\dot x,
\end{cases}
\end{equation}
where $V:\R^n\to \R$ represents the objective function, and $A(x,t)\in \R^{n\times n}$ is a positive semi-definite matrix that tunes the generalized viscosity friction. In \cite{OC} it was noticed that \textit{adaptive restart strategies} can further accelerate the convergence to the minimizer, since they are capable of eliminating the oscillations typical of under-damped mechanical systems. The term \textit{adaptive restart} denotes a procedure that resets to $0$ the momentum/velocity variable (i.e., $p$ in \eqref{eq:diss_ODE}), as soon as a suitable condition is satisfied.
In \cite{ScCF22} it was considered a \textit{conservative} dynamics by dropping the viscosity term, i.e, choosing $A(x,t)\equiv 0$ in \eqref{eq:diss_ODE}. Then, using the symplectic Euler scheme (see, e.g., \cite{Hairer}) to discretize the system, it was proposed the following \textit{conservative} algorithm:
\begin{equation}\label{eq:cons_meth_reg}
\begin{cases}
p^{k+1} = p^k -h_m \nabla V(x^k),\\
x^{k+1} = x^k +h_m p^{k+1},
\end{cases}
\end{equation}
where $h_m>0$ represents the discretization step-size.
In the case of a regular and convex objective $V$, the conservative scheme \eqref{eq:cons_meth_reg} achieves at each iteration a decrease of the function $V$ greater or equal than the classical gradient descent. This fact relies on the following restart strategy: ``reset $p_k=0$ whenever $\langle \nabla V(x^{k+1}), p^k \rangle>0$''. 
In \cite{ScCF22} it was also investigated a heuristic extension of \eqref{eq:cons_meth_reg} to the case of a non-smooth objective $f:\R^n\to\R$ with $\ell_1$-composite structure, where $\partial^- f(x^k)$ was used in \eqref{eq:cons_meth_reg} in place of $\nabla V(x^k)$, i.e., 
\begin{equation}\label{eq:cons_meth_non_smooth}
\begin{cases}
p^{k+1} = p^k -h_m \partial^- f(x^k),\\
x^{k+1} = x^k +h_m p^{k+1}.
\end{cases}
\end{equation}
In this section, taking advantage of the observations done in Section~\ref{sec:subgr_meth} for the non-accelerated subgradient method, we propose a variant of the algorithm described in \cite[Algorithm~4]{ScCF22}. The main differences concern the way we manage the changes of sign in the components, and the condition for the reset of the momentum variable. Indeed, from \eqref{eq:cons_meth_non_smooth} we deduce that
\begin{equation} \label{eq:default_step_nonsmooth}
x^{k+1} = x^k -h\partial^- f(x^k) +\sqrt h p^k,
\end{equation}
where we set $h=h_m^2$. Therefore, it is natural to divide every step of the accelerated algorithm into two phases:
\begin{itemize}
\item $q\gets x^k -h\partial^- f(x^k)$ (subgradient phase). If sign changes in the components occur, we adopt the same procedures as in Algorithm~\ref{Alg_subgr_l1}.
\item $q' \gets q + \sqrt h p^k$ (momentum phase). Also in this phase, we have particular care of sign changes of the components.
\end{itemize}
Moreover, we use the general principle that ``in the \textit{momentum phase} we do not modify null components''. This is motivated by the fact that the momentum variable carries information about the previous values of the $\partial^- f$. However, since $\partial_i^- f$ typically undergoes sudden modification when the $i$-th component of the state variable $x^k$ vanishes or changes sign, the information contained in $p^k_i$ could be of little use, if not misleading.
For this reason, in Algorithm~\ref{Alg_acc_subgr_l1} we set $p_i^k=0$ if the $i$-th component of the state variable is  null, or if it has been involved in a sign change. See, respectively, line~10 and line~17 of the accelerated subgradient method reported in Algorithm~\ref{Alg_acc_subgr_l1}.
Finally, in virtue of \eqref{eq:default_step_nonsmooth} and the remarks done above, we observe that a natural choice for the stepsize is $h=1/L$, where $L$ is the Lipschitz constant of the gradient of the regular term $g:\R^m\to\R$.

\begin{algorithm}
\caption{Accelerated conservative subgradient method for $\ell_1$-composite optimization}
\begin{algorithmic}[1]
  \scriptsize
  \STATE $x \gets x^0$, $p\gets 0\in \R^n$, $h>0$ constant step-size
  \STATE $k \gets 0$
  \WHILE{$k \leq \max_{\mathrm{iter}}$}
  \STATE $x^{\mathrm{temp}} \gets x^k - h\partial^- f(x^k)$
  \IF {$\sign(x_i^{\mathrm{temp}}\cdot x_i^k )\geq 0, \forall i =1,\ldots,n$}
  \STATE $q^{\mathrm{old}}\gets x^k$
  \STATE $q\gets \xtemp$
  \STATE $p_i \gets 0,$ if $\xtemp_i=0$
  \ELSE
  \STATE $I\gets \{ i \mid \sign(x_i^{\mathrm{temp}}\cdot x_i^k )\leq 0 \}$
  \STATE $x'_i\gets x_i^k, \quad \forall i\not\in I$
  \STATE $x'_i\gets 0, \quad \forall i\in I$
  \STATE $v_i'' \gets -h\partial_i^-f(x^k), \quad \forall i\not\in I$
  \STATE $v_i'' \gets -h\partial_i^-f(x'), \quad \forall i\in I$
  \STATE $p_i\gets 0, \quad \forall i\in I$ 
  \STATE $x''\gets x' + v''$
  \STATE $q^{\mathrm{old}}\gets x'$
  \IF{$f(x')<f(x'')$}
  \STATE $q\gets x'$
  \STATE $p\gets 0$
  \ELSE
  \STATE $q\gets x''$
  \ENDIF
  \ENDIF
  \STATE $q' \gets q + \sqrt h p$
  \IF{$\exists i=1,\ldots,n: \sign(q'_i\cdot q_i )< 0$}
  \STATE $J\gets \{ i\mid \sign(q'_i\cdot x_i )< 0 \}$
  \STATE $q'_i\gets 0,\quad \forall i\in J$
  \STATE $p\gets (q'-q)/\sqrt{h}$
  \ENDIF
  \STATE $r\gets \langle \widetilde \partial f(q'), p \rangle$
  \IF{$r\leq 0$}
  \STATE $p\gets p + (q-q^{\mathrm{old}})/\sqrt{h}$
  \ELSE
  \STATE $q'\gets q$
  \STATE $p\gets (q-q^{\mathrm{old}})/\sqrt{h}$
  \ENDIF
  \STATE $x^{k+1}\gets q'$
  \STATE $k \gets k+1$
  \ENDWHILE
\end{algorithmic}
\label{Alg_acc_subgr_l1}
\end{algorithm}

\begin{remark}
In line~31 of Algorithm~\ref{Alg_acc_subgr_l1} we have introduced the quantity $\widetilde \partial f(q')$. We recall that $f(x)=g(x)+\gamma |x|_1$, where $g$ is convex and $C^1$-regular, and $\gamma>0$. Using the same notations as in Algorithm~\ref{Alg_acc_subgr_l1}, $\widetilde \partial f(q')=(\widetilde \partial_1 f(q'),\ldots, \widetilde \partial_n f(q'))$ is defined as follows:
\begin{equation}\label{eq:def_dir_der_acc}
\widetilde \partial_i f(q'):=
\begin{cases}
\partial_i g(q') + \gamma &\mbox{if } (q_i>0) \lor (q'_i>0),\\
\partial_i g(q') + \gamma &\mbox{if } (q_i<0) \lor (q'_i<0),\\
\partial_i g(q') &\mbox{if } (q_i=0) \land (q'_i=0),
\end{cases}
\end{equation}
for every $i=1,\ldots,n$. We observe that $\widetilde \partial f(q')$ is well-defined for every component since, by construction, $\sign(q_i\cdot q'_i)\geq 0$ for every $i=1,\ldots,n$.
\end{remark}

\begin{remark}
We observe that the computation of the quantity $r$ at the line~35 requires an evaluation of the subdifferential of $f$ at the point $q'$. From a computational viewpoint, the demanding part is the evaluation of the gradient of the regular term, i.e., $\nabla g(q')$. However, if $r\leq 0$, then $x\gets q'$ (line~42), and $\nabla g(q')$ can be stored and re-used for the construction of $\partial^- f(x)$ at the subsequent iteration. 
\end{remark}

We can prove the following result on the decrease of the objective function $f$, guaranteeing that, in any circumstance, Algorithm~\ref{Alg_acc_subgr_l1} is at least as good as Algorithm~\ref{Alg_subgr_l1}.

\begin{proposition}\label{prop:dec_acc_meth}
Let $f:\R^n\to\R$ be a function such that $f(x) =g(x) +\gamma |x|_1$ for every $x\in \R^n$, where $\gamma>0$ and $g:\R^n\to\R$ is  a $C^1$ convex function such that $\nabla g$ is $L$-Lipschitz continuous, with $L>0$.
Let us consider $q^0\in \R^n$ as the initial point, and let $q'$ be the output produced by an iteration of  Algorithm~\ref{Alg_acc_subgr_l1}  and let $q$ be the output of an iteration of Algorithm~\ref{Alg_subgr_l1}  (see line 29 of  Algorithm~\ref{Alg_acc_subgr_l1}).
Then, we have that $f(q')\leq f(q)$.
\end{proposition}

\begin{remark}
Under the same assumptions as Theorem~\ref{thm:lin_conv_subg}, i.e., when $g:\R^n\to\R$ is $\mu$-strongly convex, from Proposition~\ref{prop:dec_acc_meth} it follows that Algorithm~\ref{Alg_acc_subgr_l1} achieves a linear convergence rate.
Indeed, if we denote by $(x^k)_{k\geq0}$ the sequence generated by Algorithm~\ref{Alg_acc_subgr_l1} setting the step-size $h$ equal to the inverse of the Lipschitz constant of $\nabla g$, then, if we apply Proposition~\ref{prop:dec_acc_meth} with $q^0=x^k$, for every $k\geq0$ we have:
\begin{equation*}
f(x^{k+1})-f(x^*) \leq f(q)-f(x^*) \leq \kappa \left( f(x^k) -f(x^*) \right),
\end{equation*}
where $\kappa\in (0,1)$ is the constant appearing in Theorem~\ref{thm:lin_conv_subg}, and $q\in\R^n$ is the output of a single iteration of Algorithm~\ref{Alg_subgr_l1} with starting point $x^k$.
\end{remark}

\begin{proof}
Using the same notations as in Algorithm~\ref{Alg_acc_subgr_l1}, we have that $q$ is obtained from $q^0$ with an iteration of Algorithm~\ref{Alg_subgr_l1} (see line~19 and line~22 of Algorithm~\ref{Alg_acc_subgr_l1}). If $p=0$, then there is nothing to prove.
On the other hand, owing to the \texttt{if} statement at lines~26-30, we have that $\sign(q'_i\cdot q_i)\geq 0$ for every $i=1,\ldots, n$. We further observe that $q'=q+\sqrt h p$ holds in every case (see line~25 and line~29). Let us define 
\begin{align*}
\xi^+ & = \{ i=1,\ldots,n\mid (q_i>0)\lor
(q'_i>0) \},\\
\xi^- & = \{ i=1,\ldots,n\mid (q_i<0)\lor
(q'_i<0) \},\\
\xi^0 & = \{ i=1,\ldots,n\mid (q_i=0)\land
(q'_i=0) \},
\end{align*}
and the set 
\begin{equation*}
Z_{\xi^\pm,\xi^0}:=\{z\in \R^n\mid 
z_i\geq 0 \mbox{ if }i\in \xi^+,
z_i\leq 0 \mbox{ if }i\in \xi^-,
z_i= 0 \mbox{ if }i\in \xi^0
\}.
\end{equation*}
Then, we have that $q,q'\in Z_{\xi^\pm,\xi^0}$, and that the restrictions $f|_{Z_{\xi^\pm,\xi^0}}\equiv \tilde f|_{Z_{\xi^\pm,\xi^0}}$, where $\tilde f:\R^n\to\R$ is a $C^1$-regular and convex function that satisfies:
\begin{equation*}
z\mapsto \tilde f(z) = g(z) + \gamma \sum_{i\in \xi^+}z_i -\gamma \sum_{i\in \xi^-}z_i.
\end{equation*}
Moreover, from \eqref{eq:def_dir_der_acc} we read that $\nabla \tilde f (q') = \widetilde \partial f(q')$. 
Since $\tilde f$ is convex, we have that 
\begin{equation*}
\tilde f(q)\geq \tilde f(q') + \langle \nabla \tilde f(q') , -\sqrt h p \rangle
\end{equation*}
and, recalling that $f(q) = \tilde f(q)$ and $f(q') = \tilde f(q')$, it follows that the condition $\langle \nabla \tilde f(q') , p \rangle\leq 0$ implies $f(q')\leq f(q)$.\\
On the other hand, if $\langle \nabla \tilde f(q') , p \rangle > 0$, then we reset $q'=q$ (see line~35), and $f(q')= f(q)$.
\end{proof}

\section{Numerical experiments} \label{sec:num_exp}

In this section we present some numerical experiments involving composite objective functions with $\ell_1$-regularization. 
We tested Algorithm~\ref{Alg_subgr_l1} and its accelerated version Algorithm~\ref{Alg_acc_subgr_l1} on objective functions of the form $f(x) = g(x) + \gamma |x|_1$, where $g:\R^n\to\R$ is convex and regular. We considered both the strongly convex and the non-strongly case. For each class of problems, we compared the performances of our methods with ISTA, i.e., the standard \textit{forward-backward} thresholding algorithm for $\ell_1$-regularized problems (see, e.g., \cite{CW05}). In \cite{BT} an accelerated version of ISTA (called Fast ISTA, or FISTA) was proposed, and in \cite{OC} it was observed that the convergence rate of FISTA can be further improved by means of adaptive restarts. We use the restarted FISTA described in \cite{OC} as the benchmark for the experiments of this part. We also reported the performances of the conservative-restart algorithm introduced in \cite{ScCF22}. The results are illustrated in Figure~\ref{fig:experiments_l1}.

\subsubsection*{Quadratic function with $\ell_1$-regularization} We considered a function $f:\R^n\to\R$ of the form
\begin{equation*}
f(x) = \frac12 x^T M x + b^Tx + \gamma |x|_1,
\end{equation*}
where $M\in \R^{n\times n}$ is a symmetric positive definite matrix with eigenvalues sampled uniformly in the interval $[0.02,100]$, and $b\in \R^n$ was generated with a Gaussian distribution $\mathcal{N}(0,4)$. We set $\gamma = 0.25\, |b|_\infty$, and we sampled the starting point using $\mathcal{N}(0,2)$. We fixed the dimension $n=1000$.
We observe that the objective function $f$ is strongly convex, hence, in principle, it could be possible to consider optimization schemes designed for strongly convex problems. However, their efficiency relies on how sharp is the available estimate of the strong convexity constant. On the other hand, both restarted-FISTA and Algorithm~\ref{Alg_acc_subgr_l1} do not require this information. This is one of the features of restarted-FISTA highlighted in \cite{OC}.

\subsubsection*{Quadratic regression with $\ell_1$-regularization}
We considered a \textit{sparse} quadratic regression problem. We generated a sparse random vector $y\in \R^n$ whose components were non-zero with probability $p=0.3$. These values were sampled using a uniform distribution over $[0,1]$.
We took a matrix $A\in \R^{m\times n}$ whose singular values were uniformly sampled in $[1,10]$, and we set $b= Ay + w$, where $w\in \R^m$ represented a Gaussian noise distributed as $\mathcal{N}(0,0.1)$. Finally, the objective function had the form
\begin{equation*}
f(x) = \frac12 |Ax - b|_2^2 + \gamma|x|_1,
\end{equation*}
with $\gamma = 1$. We used $n=1000$ and $m=500$, and we sampled the component of the initial guess with $\mathcal{N}(0,2)$. This problem is non-strongly convex, since the matrix $M = A^T A$ has not full rank.

\subsubsection*{Logistic regression with $\ell_1$-regularization}
We considered a sparse logistic regression problem. We constructed $x^{\mathrm{real}}\in \R^n$ with the following procedure: each component was zero with probability $p=0.8$, and, if nonzero, its value was sampled using a standard normal $\mathcal{N}(0,1)$. Then, we independently sampled the entries of $b=(b_1,\ldots,b_m)\in \{ 0,1 \}^m$ using the distribution: $\mathbb{P}(b_i=1)=(1+\exp(\langle M_i, x^{\mathrm{real}} \rangle))$ for every $i=1,\ldots,m$, where $M_1,\ldots,M_m\in \R^n$ are the rows of a matrix $M\in \R^{m\times n}$ with independent components generated with $\mathcal{N}(0,1)$.
Supposing to know the matrix $M$ and the measurements $b$, the sparse log-likelyhood maximization can be formulated as the problem of minimizing 
\begin{equation*}
f(x) = g(x) + \gamma |x|_1, \quad g(x) =
\sum_{i=1}^m\Big[
(1-b_i)\langle M_i, x\rangle + \log(1+\exp(-\langle M_i, x\rangle ))
\Big],
\end{equation*}
where we set $\gamma=0.25|\nabla g(0)|_\infty$. We used $n=100$ and $m=500$, and we sampled the component of the initial guess with $\mathcal{N}(0,2)$. This problem is convex but not strongly convex.

\subsubsection*{LogSumExp with $\ell_1$-regularization}
We considered the function $f:\R^n\to\R$ defined as follows:
\begin{equation*}
f(x) = g(x) + \gamma |x|_1, \quad g(x) =
r \log \left(\sum_{i=1}^k
\exp\left( \frac{\langle M_i, x\rangle - b_i}{r} \right)\right),
\end{equation*}
where $M_1,\ldots,M_k\in \R^n$ are the rows of the matrix $M\in \R^{k\times n}$, and $b\in \R^k$. The entries of $M$ and $b$ were independently sampled using a Gaussian $\mathcal{N}(0,1)$, as well as the components of the starting point. We set $r=5$, and we used $n=200$ and $k=500$. This is another example of non-strongly convex problem.

\begin{figure}
\centering
\includegraphics[scale=0.3]{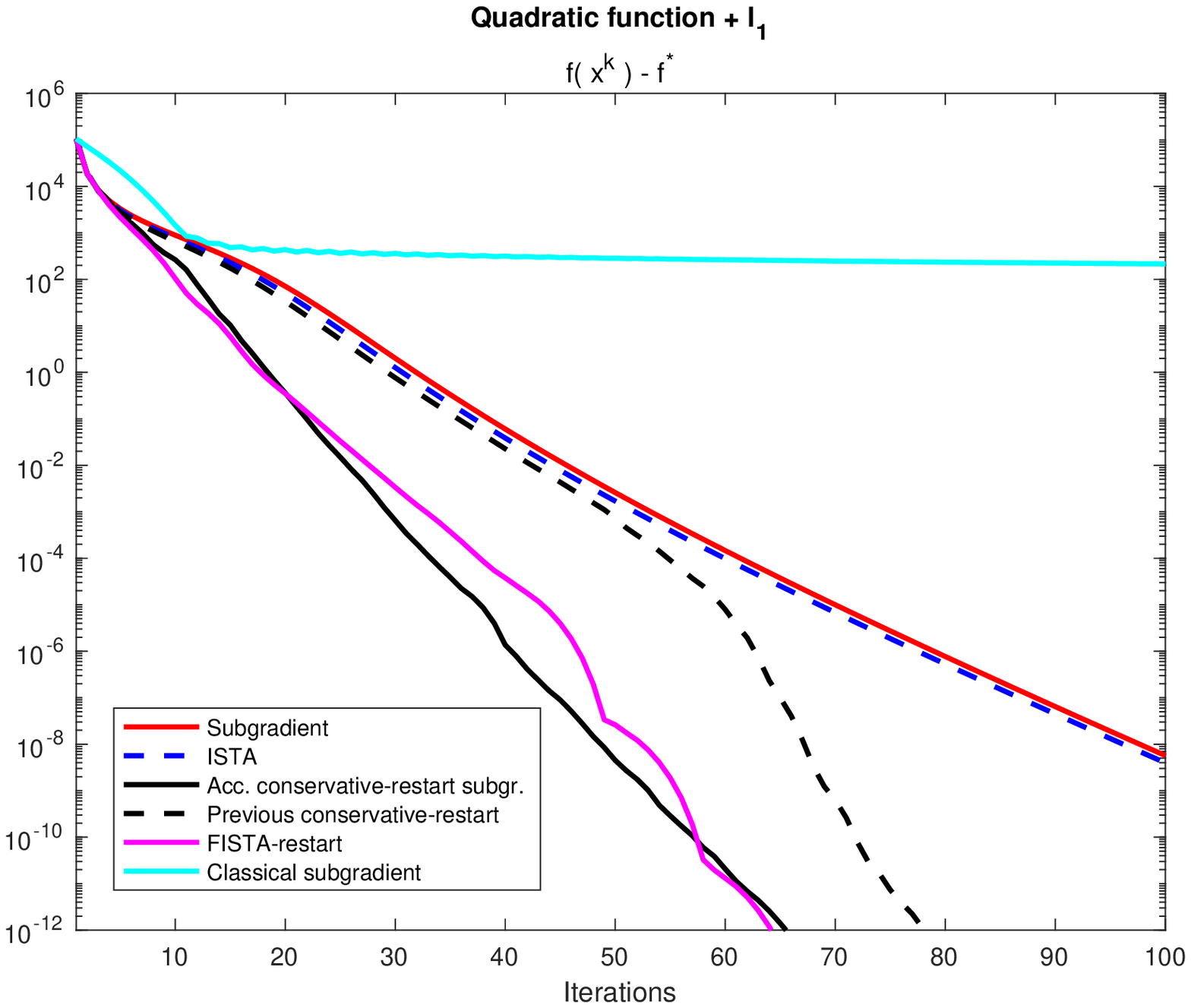}
\includegraphics[scale=0.3]{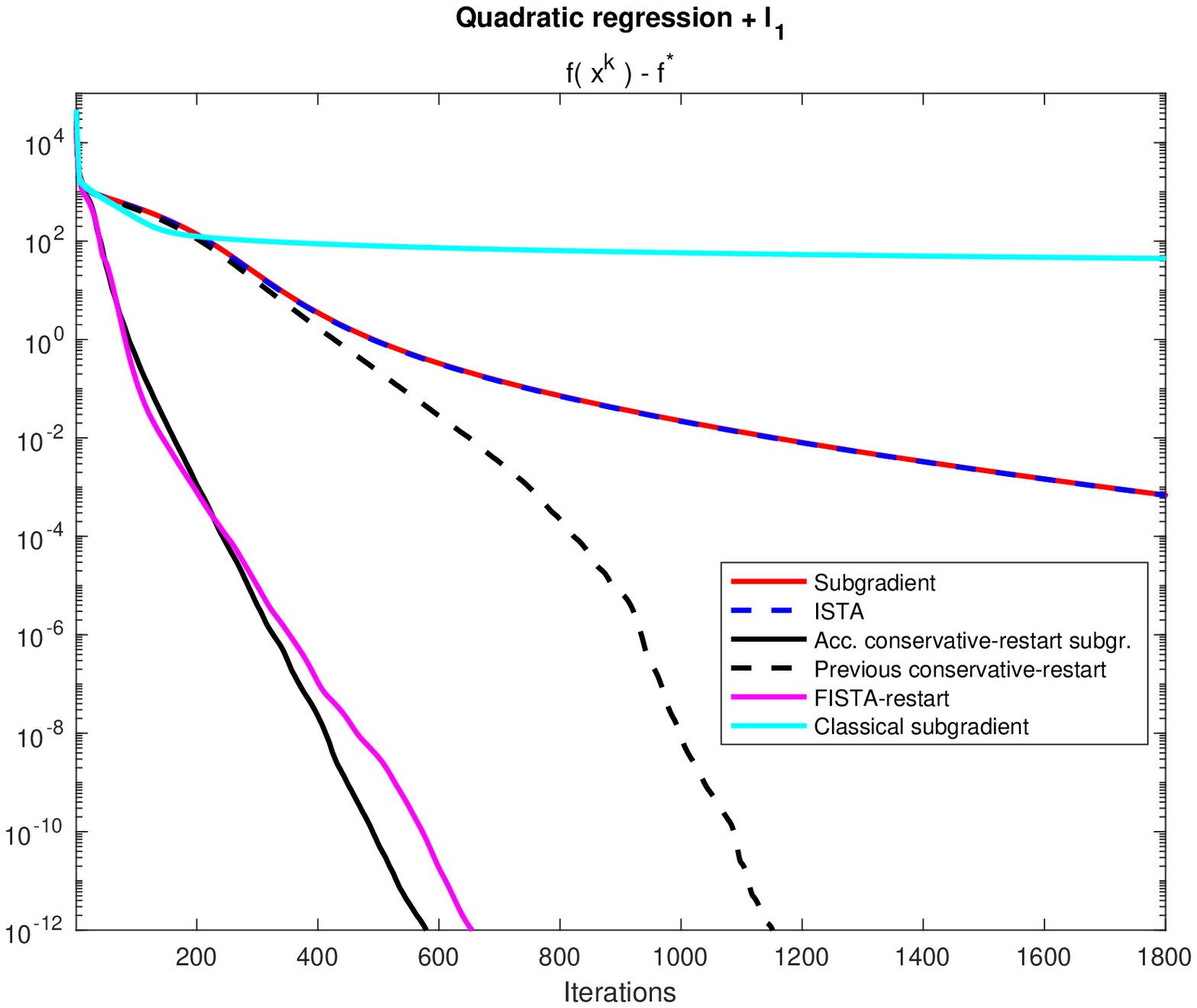}\\
\includegraphics[scale=0.3]{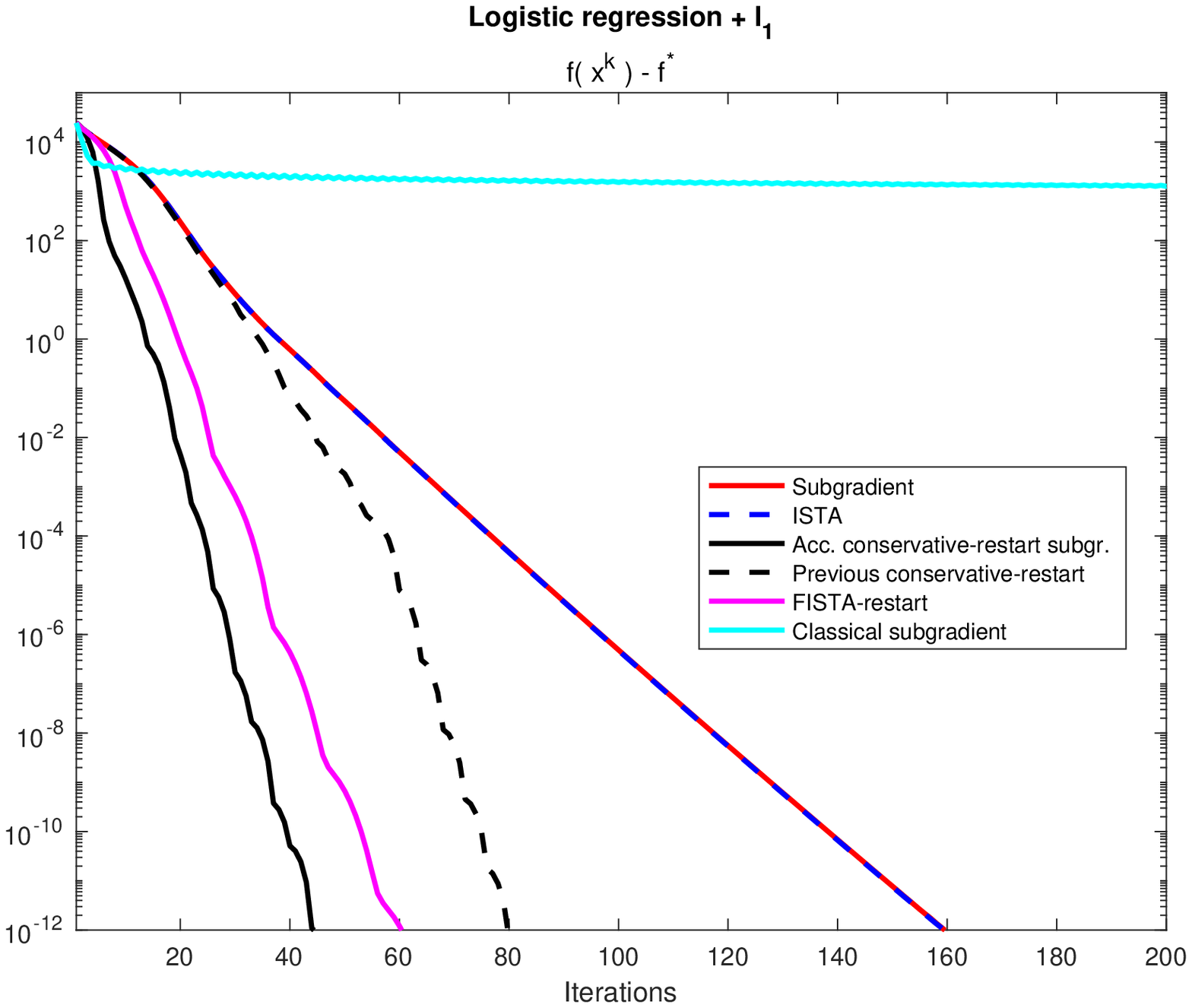}
\includegraphics[scale=0.3]{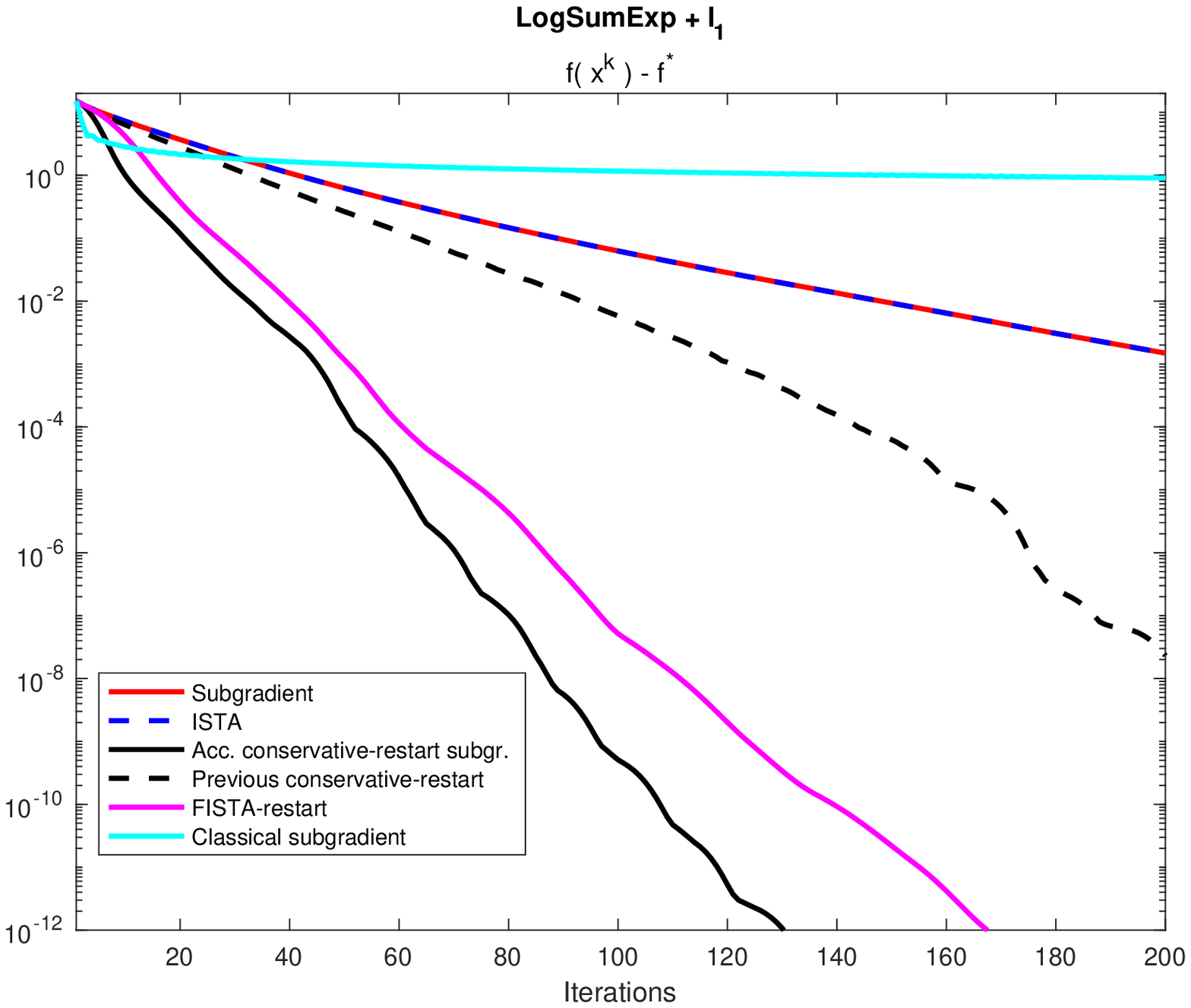}
\caption{Experiments with $\ell_1$-regularization. We compared the performances of Algorithm~\ref{Alg_subgr_l1} (red) and Algorithm~\ref{Alg_acc_subgr_l1} (black) with ISTA (dashed blue), restarted FISTA (magenta), the conservative-restart scheme proposed in \cite{ScCF22} (dashed black), and the classical subgradient method with stepsize $h_k =10k^{-1/4}$. Each problem was solved $100$ times, and the plots were obtained by taking the average. The convergence rate is measured by evaluating the gap $f(x^k)-f(x^*)$ at each iteration.
}
\label{fig:experiments_l1}
\end{figure}

\noindent
We briefly comment on the results of the experiments described above.
We observe that the non-accelerated algorithms, i.e., Algorithm~\ref{Alg_subgr_l1} and ISTA, have always very similar performances. Restarted FISTA is the most performing in the strongly convex case, while Algorithm~\ref{Alg_acc_subgr_l1} seems to be the most efficient with non-strongly convex objectives. If compared to the restart-conservative of \cite{ScCF22}, we observe that Algorithm~\ref{Alg_acc_subgr_l1} is much faster in the early phases of the minimization process. Finally, the classical subgradient method with diminishing step-size is the less performing scheme.\\

\noindent
The fact that the decays achieved Algorithm~\ref{Alg_subgr_l1} and ISTA are almost identical motivated us to construct an example where the difference in performances could be more apparent.
We considered a two-dimensional function such that $x^*=(1,0)$, and such that $\partial f(x^*) = \{0\}\times [0,\epsilon]$, for some $\epsilon>0$.
More precisely, we defined $f:\R^2 \to \R$ as
\begin{equation} \label{eq:2d_example}
f(x_1,x_2) = \frac12 \left( 
x_1^2 + 2cx_1 x_2 + 1.5x_2^2
\right) -2x_1 + (1-c)x_2 + \gamma ||x||_1,
\end{equation}
and we set $c=0.85$ and $\gamma = 1$. In this case, the  correct individuation of the fact that the second component of the minimizer is null can be challenging. This is due to the identity $\partial_2 f(x^*) = [0,\epsilon]$, or, in other words, since the vector $0$ does not lie in the relative interior of $\partial f(x^*)$. In this scenario, in the case a crossing of the set $\{ x\in \R^2: x_2 = 0 \}$ occurs, we expect that Algorithm~\ref{Alg_subgr_l1} might better decide whether the component $x_2$ should be set equal to $0$.
We used as initial guess the point $x^0=(0.95,0.5)$.
We also considered a family of problems obtained by perturbing $c$, $\gamma$  and $x^0$ with Gaussian noise
of standard deviations, respectively, equal to $0.1$, $0.1$ and $0.05$.
The results are reported in Figure~\ref{fig:experiments_example}. 
We observe that Algorithm~\ref{Alg_subgr_l1} achieves better performances than ISTA on the designed problem, and this advantage seems to be robust with respect to the perturbations introduced. 
Finally, despite using step-sizes that decay faster than in the previous experiments, the classical subgradient method exhibits evident oscillations, both in the original and in the noisy problem.

\begin{figure}
\centering
\includegraphics[scale=0.3]{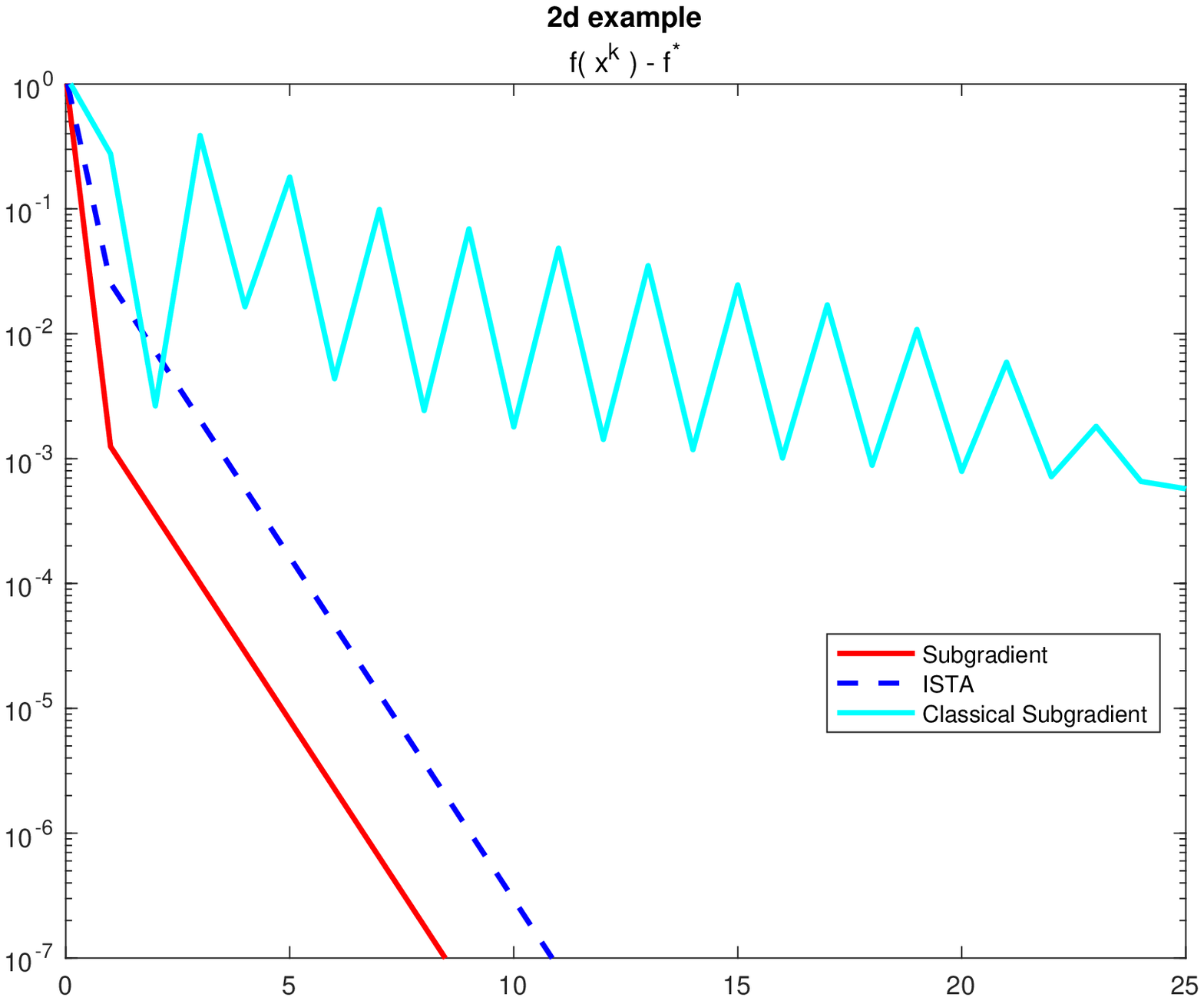}
\includegraphics[scale=0.3]{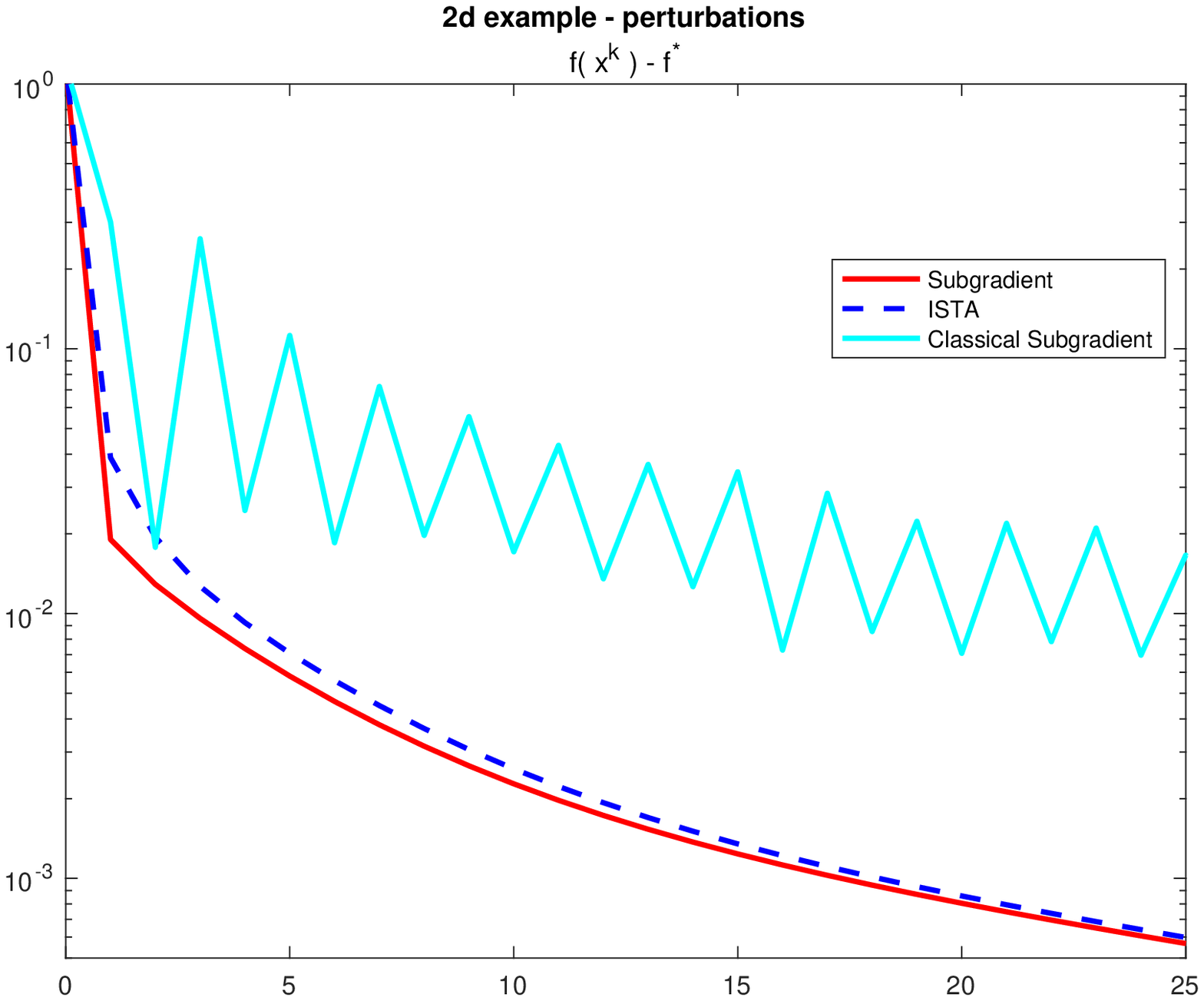}
\caption{ Comparison of 
Algorithm~\ref{Alg_subgr_l1} (red) with ISTA (dashed blue), and the classical subgradient method with stepsize $h_k =k^{-1}$.
On the left, we considered the function defined in \eqref{eq:2d_example}, while on the right we solved a perturbed problem. For the second picture, we repeated the test $100$ times, and the plots were obtained by taking the average. The convergence rate is measured by evaluating the gap $f(x^k)-f(x^*)$ at each iteration.}
\label{fig:experiments_example}
\end{figure}

\section*{Conclusions}
In this paper, we considered \textit{composite} convex optimization problems with $\ell_1$-penalization, and we formulated a subgradient algorithm with constant step-size. In the case of strongly convex objectives, we established a linear convergence result for the method.
Using dynamical system considerations, we proposed an accelerated version of the subgradient algorithm, that, at each iteration, achieves a decay of the objective always greater or equal than the decay corresponding to a step of the non-accelerated subgradient method.  
We observed in numerical experiments that the inertial  algorithm can effectively compete with one of the most performing schemes for this kind of problems, i.e., FISTA combined with an adaptive restart strategy.\\
For future work, it could be interesting to design subgradient algorithms for composite optimization involving a non-smooth term of the form $x\mapsto |Ax|_1$. In this case, a challenging point consists in finding strategies for computing $\partial^-f$ (or a suitable approximation) that could be practical for high-dimensional settings.

\subsection*{Acknowledgments}
This paper is dedicated to the beloved memory of Prof. Piero Colli Franzone.
A.S. acknowledges partial support from INdAM-GNAMPA.
A.S. wants to thank two anonymous Referees for the helpful comments that contributed to improve the quality of the paper.


\begin{thebibliography}{22}



\bibitem{A16} H. Attouch, J. Peypouquet, P. 
Redont.
{Fast convex optimization via inertial dynamics 
with Hessian driven damping.}
{\it Journal of Differential Equations},
261:5734--5783, 2016.
doi: 10.1016/j.jde.2016.08.020



\bibitem{A18} H. Attouch, Z. Chbani, J. Peypouquet, P. Redont.
{Fast convergence of inertial dynamics and
algorithms with asymptotic vanishing viscosity.}
{\it Math. Program.},
168:123--175, 2018.
doi: 10.1007/s10107-016-0992-8

\bibitem{Ber15}
D. Bertsekas.
\textit{Convex Optimization Algorithms.}
Athena Scientific, Nashua, 2015.



\bibitem{BT} A. Beck, M. Teboulle.
{A fast iterative shrinkage-thresholding algorithm for linear inverse problems.}
{\it SIAM J. Imaging Sci.},
2:183--202, 2009.
doi: 10.1137/080716542

\bibitem{BNP17}
J. Bolte, T.P. Nguyen, J. Peypouquet, B.W. Suter.
From error bounds to the complexity of first-order descent methods for convex functions. 
\textit{Math. Program.}, 165:471--507, 2017. 
doi: 10.1007/s10107-016-1091-6


\bibitem{CRT08}
E. Cand\`es, J.K. Romberg, T. Tao.
Stable signal recovery from incomplete and inaccurate measurements. 
\textit{Comm. Pure Appl. Math.}, 59: 1207--1223, 2008. 
doi: 10.1002/cpa.20124



\bibitem{CW05}
P.L. Combettes, V.R. Wajs.
{Signal recovery by proximal forward-backward splitting.}
\textit{Multiscale Model. Sim.},
 4(4):1168–1200, 2005.
doi: 10.1137/050626090

\bibitem{DDMP18}
D. Davis, D. Drusvyatskiy, K.J. MacPhee, C. Paquette. 
Subgradient Methods for Sharp Weakly Convex Functions. \textit{J. Optim. Theory Appl.}, 179: 962--982, 2018. doi: 10.1007/s10957-018-1372-8

\bibitem{ED}
F. Emmert-Streib, M. Dehmer.
{High-Dimensional LASSO-Based Computational Regression Models: Regularization, Shrinkage, and Selection.}
{\it  Mach. Learn. Knowl. Extr.}, 1:359--383, 2019.
doi: 10.3390/make1010021


\bibitem{Hairer} E. Hairer, C. Lubic, G. Wanner.
\textit{Geometric Numerical Integration:
Structure-Preserving Algorithms for Ordinary Differential Equations.}
{Springer-Verlag Berlin Heidelberg,}
2006.


\bibitem{JM20}
P.R. Johnstone, P. Moulin. 
Faster subgradient methods for functions with Hölderian growth. 
\textit{Math. Program.}, 180: 417--450, 2020.
doi: 10.1007/s10107-018-01361-0

\bibitem{JL23}
C. Josz, L. Lai.
Lyapunov stability of the subgradient method with constant step size.
\textit{Math. Program.},  2023.
doi: 10.1007/s10107-023-01936-6

\bibitem{N83} Y. Nesterov.
{A method of solving a convex programming problem with convergence rate $O(1/k^2)$.}
{\it Soviet Mathematics Doklady,}
27:372--376, 1983.

\bibitem{N13} Y. Nesterov.
{Gradient methods for minimizing composite
functions.}
{\it Math. Program.},
140:125--161, 2013.
doi: 10.1007/s10107-012-0629-5

\bibitem{N18} Y. Nesterov.
{\it Lectures on Convex Optimization.}
{Springer Nature Switzerland AG},
2018. 
doi: 10.1007/978-3-319-91578-4


\bibitem{OC} B. O'Donoghue, E. Cand\`es.
{Adaptive restart for accelerated gradient
schemes.}
{\it Found. Comput. Math.},
15(3):715--732, 2015.
doi.org/10.1007/s10208-013-9150-3

\bibitem{P63}
B.T. Polyak.
{Gradient method for the minimization of 
functionals.}
{\it USSR Comput. Math. \& Math. Phys.},
3(4):864--878, 1963.


\bibitem{P64}
B.T. Polyak.
{Some methods of speeding up the convergence of iteration methods.}
{\it USSR Computational Mathematics and 
Mathematical Physics},
4(5):1--17, 1964.


\bibitem{Po87}
B.T. Polyak.
{\it Introduction to optimization.}
{Optimization Software},
1987.


\bibitem{ReCa}
S. Rebegoldi, L. Calatroni.
Scaled, inexact and adaptive generalized FISTA for strongly convex optimization.
\textit{SIAM J. Optim.},  32(3):2428--2459, 2022. 
doi: 10.1137/21M1391699


\bibitem{R97}
R.T. Rockafellar.
{\it Convex Analysis.}
{Princeton University Press},
1997.


\bibitem{ScCF22}
A. Scagliotti, P. Colli Franzone.
A piecewise conservative method for unconstrained convex optimization.
\textit{Comput. Optim. Appl.}, 81:251–288, 2022.
doi: 10.1007/s10589-021-00332-0


\bibitem{SJ19} B. Shi, S.S. Du, M.I. Jordan, W.J. Su.
{Acceleration via symplectic discretization of high-resolution differential equations}
{\it Advances in Neural Information Processing Systems}, 
32:5744--5752, 2019.



\bibitem{Sh68}
N.Z. Shor. 
The rate of convergence of the generalized gradient descent method. 
\textit{Cybern. Syst. Anal.} 4:79–80, 1968. doi: 10.1007/BF01073933

\bibitem{Sh98}
N.Z. Shor.
\textit{Nondifferentiable Optimization and Polynomial Problems.}
Springer Science+Business Media Dordrecht, 1998.
doi: 10.1007/978-1-4757-6015-6


\bibitem{SBC} W.J. Su, S. Boyd, E. Cand\`es.
{A differential equation for modeling Nesterov’s
accelerated gradient method: 
theory and insights.} 
{\it Journal of Machine Learning Research}, 
17(153):1--43, 2016.

\bibitem{VBL}
D. Vidaurre, C. Bielza, and P. Larranaga.
{A Survey of L1 Regression}
{\it International Statistical Review} , 81(3):361--387, 2013.
doi: 10.1111/insr.12023


\bibitem{XY13}
Y. Xu, W. Yin.
A Block Coordinate Descent Method for Regularized Multiconvex Optimization with Applications to Nonnegative Tensor Factorization and Completion
\textit{SIAM J. Imaging Sci.},
6(3), 2013.
doi: 10.1137/120887795


\bibitem{YW10}
J. Yang, J. Wright, T.S. Huang, Y. Ma.
Image super-resolution via sparse representation. 
\textit{IEEE Trans. Image Process.}, 19(11):2861–-2873, 2010.
doi: 10.1109/TIP.2010.2050625


\end{thebibliography}
\end{document}